\documentclass{amsart} 

\usepackage[english]{babel}
\usepackage[latin1]{inputenc}

\usepackage{amsmath}
\usepackage{amsthm}
\usepackage{amssymb}
\usepackage{tikz}

\usepackage{imakeidx} 
\usepackage{appendix}
\usepackage{tikz-cd}
\usepackage{verbatim}
\usepackage{enumitem}
\usepackage{multicol}

\usepackage[backref=page]{hyperref}
\usepackage{cleveref}

\hypersetup{colorlinks=true,linkcolor=blue,anchorcolor=blue,citecolor=blue}

\usepackage{adjustbox}

\usepackage{mathtools}

\newtheorem{thm}{Theorem}[section]
\newtheorem{prop}[thm]{Proposition}
\newtheorem{lemma}[thm]{Lemma}
\newtheorem{cor}[thm]{Corollary}
\newtheorem{conj}[thm]{Conjecture}
\newtheorem{question}[thm]{Question}

\theoremstyle{definition}

\newtheorem{example}[thm]{Example}

\theoremstyle{remark}
\newtheorem{rmk}[thm]{Remark}
\numberwithin{equation}{section}

\newcommand{\Q}{\mathbb Q}
\newcommand{\F}{\mathbb F}

\newcommand{\Z}{\mathbb Z}

\newcommand{\cl}{\overline}
\newcommand{\set}[1]{\left\{#1\right\}}
\renewcommand{\phi}{\varphi}

\newcommand{\on}[1]{\operatorname{#1}}
\newcommand{\ang}[1]{\langle{#1}\rangle}

\title{On the Massey Vanishing Conjecture and Formal Hilbert 90}

\address{Department of Mathematics\\
	University of California\\
	Los Angeles, CA 90095\\ 
	United States of America}

\author{Alexander Merkurjev}
\email{merkurev@math.ucla.edu}

\author{Federico Scavia}
\email{scavia@math.ucla.edu}

\makeatletter
\@namedef{subjclassname@2020}{\textup{2020} Mathematics Subject Classification}
\makeatother

\date{August 2023}


\begin{document}

	\begin{abstract}
	Let $p$ be a prime number, let $G$ be a profinite group, let $\theta\colon G\to \Z_p^{\times}$ be a continuous character, and for all $n\geq 1$ write $\Z/p^n\Z(1)$ for the twist of $\Z/p^n\Z$ by the $G$-action. Suppose that $(G,\theta)$ satisfies a formal version of Hilbert's Theorem 90: for all open subgroups $H\subset G$ and every $n\geq 1$, the map $H^1(H,\Z/p^n\Z(1))\to H^1(H,\Z/p\Z(1))$ is surjective. We show that the Massey Vanishing Conjecture for triple Massey products and some degenerate fourfold Massey products holds for $G$. A key step in our proof is the construction of a Hilbert 90 module for $(G,\theta)$: a discrete $G$-module $M$ which plays the role of the Galois module $F_{\text{sep}}^\times$ for the absolute Galois group of a field $F$ of characteristic different from $p$.
	\end{abstract}
	
	\maketitle

	\section{Introduction}
Let $G$ be a profinite group and $p$ be a prime number. The cohomology ring $H^*(G,\Z/p\Z)$ carries a multiplication, called cup product, as well as higher-order operations, called Massey products. When $G=\Gamma_F$ is the absolute Galois group of a field $F$ containing a primitive $p$-th root of unity, the Norm-Residue Isomorphism Theorem of Voevodsky and Rost \cite{haesemeyer2019norm} describes the cup product on $H^*(\Gamma_F,\Z/p\Z)$ by giving a presentation by generators and relations.

If $\chi_1,\dots,\chi_n\in H^1(G,\Z/p\Z)$, the Massey product of $\chi_1,\dots,\chi_n$ is a certain subset $\ang{\chi_1,\dots,\chi_n}\subset H^2(G,\Z/p\Z)$; see \Cref{massey-section} for the definition. We say that $\ang{\chi_1,\dots,\chi_n}$ is defined if it is not empty, and that it vanishes if it contains $0$. Massey products are the object of the Massey Vanishing Conjecture, due to Min\'{a}\v{c}--T\^{a}n \cite{minac2017triple} and inspired by earlier work of Hopkins--Wickelgren \cite{hopkins2015splitting}.

\begin{conj}[Min\'{a}\v{c}--T\^{a}n]\label{massey-conj}
		For every field $F$, every prime $p$, every $n\geq 3$ and all $\chi_1,\dots,\chi_n\in H^1(\Gamma_F,\Z/p\Z)$, if the Massey product $\ang{\chi_1,\dots,\chi_n}\subset H^2(\Gamma_F,\Z/p\Z)$ is defined, then it vanishes.
\end{conj}
\Cref{massey-conj} is motivated by the Profinite Inverse Galois Problem, that is, the open problem of determining which profinite groups are absolute Galois groups of fields. Indeed, the Conjecture predicts a new restriction for a profinite group to be an absolute Galois group.

\Cref{massey-conj} has attracted a lot of attention in recent years, and has been established in a number of cases. It is known when $n=3$, by Hopkins--Wickelgren \cite{hopkins2015splitting} when $F$ is a number field and $p=2$, and by work of Efrat--Matzri and Min\'{a}\v{c}--T\^{a}n \cite{matzri2014triple, efrat2017triple, minac2016triple} in general. It is also known when $n=4$ and $p=2$: when $F$ is a number field this is due to Guillot--Min\'{a}\v{c}--Topaz--Wittenberg \cite{guillot2018fourfold}, and when $F$ is arbitrary to \cite{merkurjev2023massey}. When $F$ is a number field, the conjecture has been proved for all $n\geq 3$ and all primes $p$ by Harpaz--Wittenberg \cite{harpaz2019massey}.

In this paper, we investigate the Massey Vanishing Conjecture for all profinite groups which satisfy a formal version of Hilbert's Theorem 90. As we now explain, this class of profinite groups contains all absolute Galois groups but is substantially larger.

Let $p$ be a prime number, and denote by $\Z_p$ the topological ring of $p$-adic integers, with the $p$-adic topology. Let $(G,\theta)$ be a $p$-oriented profinite group, that is, a profinite group $G$ equipped with a continuous group homomorphism $\theta\colon G\to \Z_p^{\times}$. For all $n\geq 1$, we write $\Z/p^n\Z(1)$ for the $G$-module with underlying abelian group $\Z/p^n\Z$ and where $G$ acts by multiplication via $\theta$. We say that $(G,\theta)$ satisfies formal Hilbert 90 if, for all $n\geq 1$ and all open subgroups $H\subset G$, the reduction map
\[H^1(H,\Z/p^n\Z(1))\to H^1(H,\Z/p\Z(1))\]
is surjective. In the literature, the $p$-oriented profinite groups $(G,\theta)$ satisfying formal Hilbert 90 are sometimes called $1$-cyclotomic; see e.g. \cite[Definition 4.1]{declercq2022lifting} or \cite[Definition 2.9(ii)]{blumer2023groups}.

It is a consequence of Hilbert's Theorem 90 that for every field $F$ of characteristic different from $p$, letting $\Gamma_F$ be the absolute Galois group of $F$ and $\theta_F\colon \Gamma_F\to \Z_p^\times$ be the canonical $p$-orientation (the cyclotomic character), the pair $(\Gamma_F,\theta_F)$ is a $p$-oriented profinite group which satisfies formal Hilbert 90; see \Cref{field-formal-h90}. Other profinite groups admitting a $p$-orientation satisfying formal Hilbert 90 are:  \'etale fundamental groups of semilocal $\Z[1/p]$-schemes, affine $\F_p$-schemes, and smooth curves over algebraically closed fields; see \cite[Propositions 4.9, 4.10, 4.11]{declercq2022lifting}. 

\begin{question}\label{massey-question}
		Let $p$ be a prime number, let $n\geq 3$ be an integer, and let $G$ be a profinite group. Suppose that there exists a $p$-orientation $\theta$ of $G$ such that $(G,\theta)$ satisfies formal Hilbert 90. Is it true that, for all $\chi_1,\dots,\chi_n\in H^1(G,\Z/p\Z)$, if the Massey product $\ang{\chi_1,\dots,\chi_n}\subset H^2(G,\Z/p\Z)$ is defined, then it vanishes?
\end{question}
  
 We show that the answer to \Cref{massey-question} is affirmative when $n=3$.
 
    \begin{thm}\label{main-h90}
    Let $p$ be a prime number, let $(G,\theta)$ be a $p$-oriented profinite group satisfying formal Hilbert 90 and let $\chi_1,\chi_2,\chi_3\in H^1(G,\Z/p\Z)$. The following are equivalent:
    \begin{enumerate}
        \item $\chi_1\cup\chi_2=\chi_2\cup\chi_3=0$ in $H^2(G,\Z/p\Z)$;
        \item the mod $p$ Massey product $\ang{\chi_1,\chi_2,\chi_3}$ is defined;
        \item the mod $p$ Massey product $\ang{\chi_1,\chi_2,\chi_3}$ vanishes.
    \end{enumerate}
    \end{thm}
\Cref{main-h90} confirms a conjecture of Quadrelli \cite[Conjecture 1.5]{quadrelli2021chasing}; see \Cref{comments2}.

We also show that the answer to \Cref{massey-question} is affirmative for certain degenerate mod $2$ fourfold Massey products of the form $\ang{\chi_1,\chi_2,\chi_3,\chi_1}$. The importance of these Massey products stems from the fact that they need not be defined even when $\chi_1\cup\chi_2=\chi_2\cup\chi_3=\chi_3\cup\chi_1=0$ in $H^2(G,\Z/2\Z)$; see \cite[Example A.15]{guillot2018fourfold} and \cite[Theorem 1.4]{merkurjev2022degenerate}.

    \begin{thm}\label{main-h90-degenerate}
    Let $(G,\theta)$ be a $2$-oriented profinite group satisfying formal Hilbert 90 and let $\chi_1,\chi_2,\chi_3\in H^1(G,\Z/2\Z)$. If the mod $2$ Massey product $\ang{\chi_1,\chi_2,\chi_3,\chi_1}$ is defined, then it vanishes.    
    \end{thm}

In the case of absolute Galois groups, this had been proved in \cite[Theorem 1.3]{merkurjev2022degenerate} by using quadratic forms theory (in particular, the theory of Albert forms associated to biquaternion algebras) and a number of facts on quaternion algebras split by a biquadratic extension (see \cite[Appendix A]{merkurjev2022degenerate}). \Cref{main-h90-degenerate} generalizes \cite[Theorem 1.3]{merkurjev2022degenerate} and shows that the latter can be proved using Hilbert's Theorem 90 only.    

We proved \Cref{massey-conj} for $n=4$ and $p=2$ in  \cite[Theorem 1.2]{merkurjev2023massey}. The proof used tools from algebraic geometry which go beyond Hilbert 90. For this reason, we do not know whether the result can be generalized to $2$-oriented profinite groups satisfying formal Hilbert 90.

Our proofs of Theorems \ref{main-h90} and \ref{main-h90-degenerate} are based on the notion of Hilbert 90 module. This is a $p$-divisible $G$-module $M$ whose $p$-primary torsion subgroup is isomorphic to the colimit of the $\Z/p^n\Z(1)$, where $n\geq 1$, and such that $H^1(H,M)=0$ for every open subgroup $H$ of $G$. If $F$ is a field of characteristic not $p$, the $\Gamma_F$-module $F_{\on{sep}}^\times$ is a Hilbert 90 module for $(\Gamma_F,\theta_F)$. The definition of Hilbert 90 module is quite natural but, to our knowledge, had not yet appeared in the literature. 

If $(G,\theta)$ admits a Hilbert 90 module, then it is not difficult to show that it satisfies formal Hilbert 90. It turns out that the converse is also true.

\begin{thm}\label{hilbert-90-equiv}
		Let $(G,\theta)$ be a $p$-oriented profinite group. Then $(G,\theta)$ satisfies formal Hilbert 90 if and only if it admits a Hilbert 90 module.
\end{thm}

By working with a Hilbert 90 module, we are able to generalize many familiar properties of the Galois cohomology of fields to the cohomology of $p$-oriented profinite groups satisfying formal Hilbert 90.

We conclude the Introduction with a summary of the content of each section. In \Cref{section2}, we introduce $p$-oriented profinite groups, the formal Hilbert 90 property and Hilbert 90 modules, and we prove some basic properties. In \Cref{section3}, we prove that a standard Galois cohomology sequence extends to the context of $p$-oriented profinite groups admitting a Hilbert 90 module; see \Cref{cyclicext-module}. \Cref{section4} is devoted to the proof of \Cref{hilbert-90-equiv}. The definition and basic properties of Massey products are recalled in \Cref{massey-section}. We prove \Cref{main-h90} in  \Cref{section6} and \Cref{main-h90-degenerate} in \Cref{section7}. In the final \Cref{section-example}, we illustrate \Cref{main-h90} by a simple example; see \Cref{counterexample-formal-h90}. 

\subsection*{Notation}

Let $G$ be a profinite group. A topological $G$-module is an abelian Hausdorff topological group on which $G$ acts continuously on the left. A discrete $G$-module is a topological $G$-module which carries the discrete topology. In this article, the term $G$-module will always refer to a discrete $G$-module. 

If $A$ is a $G$-module, we write $A^G$ for the subgroup of $G$-invariant elements of $A$. For all $i\geq 0$ we write $C^i(G,A)$, (resp. $Z^i(G,A)$) for the group of mod $p$ non-homogeneous degree-$i$ continuous cochains (resp. cocycles) of $G$, and we write $H^i(G,A)$ for the $i$-th cohomology group of $A$. If $H\subset G$ is an open subgroup, we write $\on{Res}^G_H\colon H^i(G,A)\to H^i(H,A)$ for the restriction maps, $\on{Cor}^G_H\colon H^i(H,A)\to H^i(G,A)$ for the corestriction maps (for $i=0$, this is the norm map $N_{G/H}\colon A^H\to A^G$), and, when $H$ is also normal, $\on{Inf}^G_{G/H}\colon H^i(G/H,A^H)\to H^i(G,A)$ for the inflation maps.

If $H$ is a closed subgroup of $G$ and $B$ is an $H$-module, we write $\on{Ind}_H^G(B)$ for the $G$-module induced by $B$: it consists of all continuous maps $x\colon G\to B$ such that $x(hg)=hx(g)$ for all $g\in G$ and $h\in H$, and $g\in G$ acts on it by the formula $(g\cdot x)(g')\coloneqq x(g'g)$ for all $g'\in G$. If $A$ is a $G$-module, we have an injective homomorphism of $G$-modules
\begin{equation}\label{ind-inclusion}i\colon A\to \on{Ind}^G_H(A),\qquad (ia)(g)=ga\text{ for all $g\in G$}
\end{equation}
and, if $H$ is an open subgroup $G$, a homomorphism of $G$-modules
\begin{equation}\label{ind-norm}
    \nu\colon \on{Ind}^G_H(A)\to A,\qquad \nu(x)=\sum_{gH\in G/H}gx(g^{-1}).
\end{equation}
For any closed subgroup $H$ of $G$, the $G$-modules $\on{Ind}_H^G(A)$ and $\on{Hom}_{\Z}(\Z[G/H],A)$ are canonically isomorphic. 

If $A$ is an abelian group, $n\geq 1$ is an integer, and $p$ is a prime number, we denote by $A[n]$ the $n$-torsion subgroup of $A$ and by $A\{p\}$ be $p$-primary torsion subgroup of $A$. If $p$ is a prime number, we write $\Z_p$ for the ring of $p$-adic integers, and $\Z_{(p)}$ for the localization of $\Z$ at $p$.

	\section{The formal Hilbert 90 property and Hilbert 90 modules}\label{section2}

\subsection{Oriented profinite groups} Let $p$ be a prime number, let $G$ be a profinite group, and let $\theta\colon G\to \Z_p^{\times}$ be a continuous group homomorphism. We call $\theta$ a {\em $p$-orientation} of $G$ and the pair $(G,\theta)$ a {\em $p$-oriented profinite group}.

We write $\Z_p(1)$ for the topological $G$-module with underlying topological group $\Z_p$ and where $G$ acts via $\theta$, that is, $g\cdot v\coloneqq \theta(g)v$ for every $g\in G$ and every $v\in \Z_p$. For all $n\geq 0$, we set $\Z/p^n\Z(1)\coloneqq \Z_p(1)/p^n\Z_p(1)$.

We say that $(G,\theta)$ is {\em cyclotomic} if the image of $\theta$ is contained in $1+p\Z_p$. The pair $(G,\theta)$ is cyclotomic if and only if the $G$-action on $\Z/p\Z(1)$ is trivial. Every $2$-oriented profinite group is cyclotomic. 

\begin{example}\label{canonical-orientation}
    Let $F$ be a field and write $\Gamma_F$ for the absolute Galois group of $F$. We define the {\em canonical $p$-orientation} $\theta_F$ on $\Gamma_F$ as follows. If $\on{char}(F)\neq p$, we define $\theta_F$ as the continuous homomorphism $\theta_F\colon \Gamma_F\to \Z_p^{\times}$ given by $g(\zeta)=\zeta^{\theta_F(g)}$ for every root of unity $\zeta$ of $p$-power order. If $\on{char}(F)=p$, we let $\theta_F$ be the trivial homomorphism. The pair $(\Gamma_F,\theta_F)$ is a $p$-oriented profinite group.
    \end{example}

\subsection{The formal Hilbert 90 property} 
Let $(G,\theta)$ be a $p$-oriented profinite group. We say that $(G,\theta)$ {\em satisfies formal Hilbert 90} if for every open subgroup $H\subset G$ and all $n\geq 1$ the reduction map $H^1(H,\Z/p^n\Z(1))\to H^1(H,\Z/p\Z(1))$ is surjective.

    \begin{example}\label{field-formal-h90}
    Let $F$ be a field, let $\Gamma_F$ be the absolute Galois group of $F$, and let $\theta_F$ be the canonical orientation on $\Gamma_F$ defined in \Cref{canonical-orientation}. We now show that $(\Gamma_F,\theta_F)$ satisfies formal Hilbert 90. Let $H\subset \Gamma_F$ be an open subgroup, and set $L\coloneqq (F_{\on{sep}})^H$.
    
    (1) Suppose that $\on{char}(F)\neq p$. By Kummer Theory, for all $n\geq 1$ the reduction map $H^1(H,\Z/p^n\Z(1))\to H^1(H,\Z/p\Z(1))$ is identified with the projection map $L^{\times}/L^{\times p^n}\to L^{\times}/L^{\times p}$, and in particular it is surjective. 

    (2) Suppose that $\on{char}(F)=p$. For all $n\geq 1$, we have $H^2(H,\Z/p^{n-1}\Z)=0$ because $L$ has cohomological $p$-dimension $\leq 1$; see \cite[\S 2.2, Proposition 3]{serre2002galois}. The short exact sequence
    \[0\to \Z/p^{n-1}\Z\to \Z/p^n\Z\to \Z/p\Z\to 0\]
    now implies the surjectivity of the reduction map $H^1(H,\Z/p^n\Z)\to H^1(H,\Z/p\Z)$.
    \end{example}

The group $\Z_p^{\times}$ acts on the abelian group $\Q/\Z_{(p)}$ by multiplication. We let $S$ be the $G$-module whose underlying abelian group is $\Q/\Z_{(p)}$ and on which $G$ acts via $\theta$. For all $n\geq 1$, we have an isomorphism of $G$-modules $\Z/p^n\Z(1)\to S[p^n]$ given by $a+p^n\Z\mapsto a/p^n+\Z_{(p)}$. Therefore, $S$ is the colimit of the $\Z/p^n\Z(1)$ for $n\geq 1$.
    
For every $n\geq 1$, the map $S\to S$ given by multiplication by $p^{n-1}$ restricts to a surjective $G$-module homomorphism $S[p^n]\to S[p]$. Then $(G,\theta)$ satisfies formal Hilbert 90 if and only if for every open subgroup $H\subset G$ and all $n\geq 1$ the induced map $H^1(H,S[p^n])\to H^1(H,S[p])$ is surjective.

 \begin{lemma}\label{surjective-induction}
     Let $(G,\theta)$ be a $p$-oriented profinite group which satisfies formal Hilbert 90. Then for every open subgroup $H\subset G$ and every $n\geq 1$ the maps 
     \[Z^1(H,S[p^n])\to Z^1(H,S[p^{n-1}]),\qquad H^1(H,S[p^n])\to H^1(H,S[p^{n-1}]).\] are surjective.
 \end{lemma}

 \begin{proof}
     We first prove the surjectivity of $H^1(H,S[p^n])\to H^1(H,S[p^{n-1}])$ by induction on $n\geq 1$. The case $n=1$ is trivial. Now let $n\geq 2$, and suppose that the statement is true for $n-1$. We have a commutative diagram with exact rows
     \[
    \begin{tikzcd}
    0 \arrow[r] & S[p^{n-1}] \arrow[r] \arrow[d,"p"] & S[p^n] \arrow[r,"p^{n-1}"] \arrow[d,"p"] & S[p] \arrow[d,equal] \arrow[r] & 0 \\
    0 \arrow[r] & S[p^{n-2}] \arrow[r] & S[p^{n-1}] \arrow[r,"p^{n-2}"] & S[p] \arrow[r] & 0.
    \end{tikzcd}
     \]
    Passing to cohomology, we obtain the diagram
    \[
\begin{tikzcd}
    H^1(H,S[p^{n-1}]) \arrow[r] \arrow[d,->>] & H^1(H,S[p^n]) \arrow[r,->>] \arrow[d] & H^1(H,S[p]) \arrow[d,equal]\\
    H^1(H,S[p^{n-2}]) \arrow[r] & H^1(H,S[p^{n-1}]) \arrow[r] & H^1(H,S[p]),
    \end{tikzcd}
    \]
    where the rows are exact, the top-right horizontal map is surjective because $(G,\theta)$ satisfies formal Hilbert 90, and the left vertical map is surjective by inductive assumption. A diagram chase now shows that the middle vertical map is surjective. The conclusion follows by induction on $n$.  

    We now prove the surjectivity of the map $Z^1(H,S[p^n])\to Z^1(H,S[p^{n-1}])$ for all $n\geq 1$. Let $\sigma\in Z^1(H,S[p^{n-1}])$. Since the map $H^1(H,S[p^n])\to H^1(H,S[p^{n-1}])$ is surjective, there exist $\tau\in  Z^1(H,S[p^n])$ and $a\in S[p^{n-1}]$ such that for all $h\in H$ we have $p\tau(h)-\sigma(h)=(h-1)a$ in $S[p^{n-1}]$. Let $b\in S[p^{n}]$ be such that $a=pb$, and define $\widetilde{\sigma}\in Z^1(H,S[p^n])$ by $\widetilde{\sigma}(h)=\tau(h)-(h-1)b$ for all $h\in H$. Then $p\widetilde{\sigma}=\sigma$, as desired.    
 \end{proof}

\subsection{Hilbert 90 modules}
   Let $(G,\theta)$ be a $p$-oriented profinite group. A {\em Hilbert 90 module} for $(G,\theta)$ is a discrete $G$-module $M$ such that
   \begin{enumerate}[label=(\roman*)]
	\item $pM=M$,
	\item $M\{p\}\simeq S$ as $G$-modules, and
	\item $H^1(H,M)=0$ for any open subgroup $H\subset G$.
	\end{enumerate}
The definition of Hilbert 90 module is modeled on the following example.

    \begin{example}\label{example-h90-galois}
    Let $p$ be a prime number, let $F$ be a field of characteristic different from $p$, let $\Gamma_F$ be the absolute Galois group of $F$, and let $\theta$ be the canonical orientation on $\Gamma_F$; see \Cref{canonical-orientation}. It follows from Hilbert's Theorem 90 that $F_{\on{sep}}^{\times}$ is a Hilbert 90 module for $(\Gamma_F,\theta_F)$.
    \end{example}

\begin{lemma}\label{subgroup}
    Let $(G,\theta)$ be a $p$-oriented profinite group, and let $M$ be a Hilbert 90 module for $(G,\theta)$. Then $M$ is also a Hilbert 90 module for $(H,\theta|_H)$ for every closed subgroup $H$ of $G$.
\end{lemma}

\begin{proof}
    Since properties (i) and (ii) of the definition of Hilbert 90 module are clearly preserved by passing to subgroups, we need only check (iii). Let $H'$ be an open subgroup of $H$. Then $H'$ is equal to the intersection of all the subgroups $H'N$, where $N$ ranges over all open normal subgroups of $G$. By \cite[Theorem 3.16]{koch2002galois}, we have
		\[
		H^1(H',M)= \underset{N}{\on{colim}}\, H^1(H'N,M)=0,
		\]
		where we have used the fact that if $N$ is open in $G$, then so is $H'N$.
\end{proof}

We conclude this section by generalizing two basic properties of Galois cohomology to profinite groups admitting a Hilbert 90 module. 

\begin{lemma}\label{omega-lemma}
    Let $(G,\theta)$ be a $p$-oriented profinite group, let $M$ a Hilbert 90 module for $(G,\theta)$, and let $a\in M^G$. For all $g\in G-G_a$, we have an exact sequence
    \[M^{G_a}\xrightarrow{g-1} M^{G_a}\xrightarrow{N_{G/G_a}} M^G.\]
\end{lemma}

This is a generalization of Hilbert's Theorem 90 for degree-$p$ cyclic extensions.

\begin{proof}
    We may assume that $G_a\neq G$. Let $C \coloneqq G/G_a$, let $g\in G-G_a$, and let $\sigma\coloneqq gG_a\in C$. Thus $C$ is a cyclic group of order $p$, generated by $\sigma$, and $M^{G_a}$ is naturally a $C$-module. By \cite[Proposition (1.2.6) and end of p. 23]{neukirch2008cohomology}, the homology of the complex in the statement is isomorphic to the Tate cohomology group $\hat{H}^{-1}(C,M^{G_a})$. It thus suffices to show that $\hat{H}^{-1}(C,M^{G_a})$ is trivial. 
    
    By \cite[Proposition (1.6.7)]{neukirch2008cohomology}, the inflation map $H^1(C,M^{G_a})\to H^1(G,M)$ is injective. Since $M$ is a Hilbert 90 module for $(G,\theta)$, the group $H^1(G,M)$ is trivial, and hence $H^1(C,M^{G_a})$ is also trivial. Since $C$ is a finite cyclic group, by \cite[Proposition (1.7.1)]{neukirch2008cohomology} the group $\hat{H}^{-1}(C,M^{G_a})$ is isomorphic to $H^1(C,M^{G_a})$, and hence it is also trivial, as desired.
\end{proof}

	Let $(G,\theta)$ be a cyclotomic $p$-oriented profinite group, and let $M$ be a Hilbert 90 module for $(G,\theta)$. We fix an injective homomorphism $\iota\colon \Z/p\Z\hookrightarrow M$. The short exact sequence
        \begin{equation}\label{fix-iota}0\to\Z/p\Z\xrightarrow{\iota} M\xrightarrow{p}M\to 0\end{equation}
        induces an exact sequence
\begin{equation}\label{connecting}
    M^G\xrightarrow{p} M^G\xrightarrow{\partial} H^1(G,\Z/p\Z)\to 0.
\end{equation}
For all $a\in M^G$, we set $\chi_a\coloneqq \partial(a)$ and $G_a\coloneqq \on{Ker}(\chi_a)$. The exactness of (\ref{connecting}) amounts to the following statement: Every $\chi\in H^1(G,\Z/p\Z)$ is of the form $\chi_a$ for some $a\in M^G$, uniquely determined up to an element of $p(M^G)$.

\begin{example}
    We maintain the notation of \Cref{example-h90-galois}. The pair $(\Gamma_F,\theta_F)$ is cyclotomic if and only if $F$ contains a primitive $p$-th root of unity. A choice of a primitive $p$-th root of unity $\zeta\in F$ determines an injection $\iota \colon \Z/p\Z\to F_{\on{sep}}^\times$. The map $\partial\colon  F^\times\to H^1(\Gamma_F,\Z/p\Z)$ of (\ref{connecting}) sends $a\in F^\times$ to the character $\chi_a\colon\Gamma_F\to \Z/p\Z$ determined by $g(a)=\zeta^{\chi_a(g)}a$ for all $a\in F_{\on{sep}}^\times$.
\end{example}

\begin{lemma}\label{serre-brauer}
    Let $(G,\theta)$ be a cyclotomic $p$-oriented profinite group, let $M$ be a Hilbert 90 module for $(G,\theta)$, and let $a,b\in M^G$. Let $\omega_a\in H^2(G,\Z)$ be the Bockstein of $\chi_a\in H^1(G,\Z/p\Z)$, that is, the image of $\chi_a$ under the connecting homomorphism $H^1(G,\Z/p\Z)\to H^2(G,\Z)$ associated to the short exact sequence of trivial $G$-modules
    \[0\to \Z\xrightarrow{\times p}\Z\to \Z/p\Z\to 0.\] Then the map $H^2(G,\Z/p\Z)\to H^2(G,M)$ sends $\chi_a\cup\chi_b$ to $b\cup \omega_a$.
\end{lemma}

    When $G$ is the absolute Galois group of a field of characteristic different from $p$, \Cref{serre-brauer} reduces to a well-known property of the Brauer group of a field; see \cite[XIV.\S 2 Proposition 5]{serre1979local}, whose proof adapts to our more general situation without essential changes.

\begin{proof}
     Let $\beta\in M$ be such that $p\beta=b$. By definition, $b\cup\omega_a$ is represented by the $2$-cocycle $G\to M$ which sends a pair $(g,h)\in G\times G$ to \[\beta\cup(\widetilde{\chi}_a(g)+\widetilde{\chi}_a(h)-\widetilde{\chi}_a(gh)),\] where $\widetilde{\chi}_a\colon G\to \Z$ is a function lifting $\chi_a$. The anticommutativity of the cup product gives $\chi_a\cup\chi_b=-\chi_b\cup\chi_a$. Using  the formula for the cup product of two $1$-cocycles of \cite[p. 221]{cartan1956homological}, the image of $-\chi_b\cup\chi_a$ in $H^2(G,M)$ may be represented by the $2$-cocycle which sends a pair $(g,h)\in G\times G$ to
    \[-(g\beta-\beta)\widetilde{\chi}_a(h).\]
    The difference of these two cocycles sends a pair $(g,h)\in G\times G$ to
    \[g\beta\cup\widetilde{\chi}_a(h)-\beta\cup\widetilde{\chi}_a(gh)+\beta\cup\widetilde{\chi}_a(g).\]
    This is the coboundary of the $1$-cochain which sends $g\in G$ to $\beta\cup\widetilde{\chi}_a(g)$.
\end{proof}

	\section{The standard exact sequence}\label{section3}

    \begin{lemma}\label{norm-cor}
        Let $(G,\theta)$ be a cyclotomic $p$-oriented profinite group, let $M$ be a Hilbert 90 module for $(G,\theta)$, and let $H$ be an open subgroup of $G$. We have a commutative square
        \[
        \begin{tikzcd}
        M^H \arrow[r,"\partial"] \arrow[d,"N_{G/H}"] & H^1(H,\Z/p\Z) \arrow[d,"\on{Cor}^G_H"] \\ 
        M^G \arrow[r,"\partial"] & H^1(G,\Z/p\Z), 
        \end{tikzcd}
        \]
        where the horizontal maps come from (\ref{connecting}).
    \end{lemma}

    \begin{proof}
    This follows from the fact that $\on{Cor}^G_H$ commutes with connecting homomorphisms (see \cite[Proposition (1.5.2)]{neukirch2008cohomology}) and the fact that $\on{Cor}^G_H$ coincides with $N_{G/H}$ in degree $0$ (see \cite[\S 5, 4. p.47]{neukirch2008cohomology}).
    \end{proof}
 
	\begin{prop}\label{cyclicext-module}
		Let $(G,\theta)$ be a $p$-oriented profinite group that admits a Hilbert 90 module. Let $\chi:G\to\Z/p\Z$ be a character and $H=\on{Ker}(\chi)$.
		Then the sequence
		\[
		H^1(H,\Z/p\Z)\xrightarrow{\on{Cor}} H^1(G,\Z/p\Z) \xrightarrow{\cup\chi} H^2(G,\Z/p\Z)\xrightarrow{\on{Res}} H^2(H,\Z/p\Z)
		\]
		is exact.
	\end{prop}

	\begin{proof}
	Let $G'\coloneqq \theta^{-1}(1+p\Z_p)$, let $\chi'\coloneqq \chi|_G$, and let $H'\coloneqq \on{Ker}(\chi')=H\cap G'$. Note that $G'$ is an open subgroup of $G$, the index $[G:G']$ divides $p-1$ and $(G',\theta|_{G'})$ is cyclotomic. Since $p$ does not divide $[G:G']$ and $[H\colon H']$, for all $i\geq 0$ the compositions
 \[H^i(G,\Z/p\Z)\xrightarrow{\on{Res}} H^i(G',\Z/p\Z)\xrightarrow{\on{Cor}}H^i(G,\Z/p\Z)\]
 and
 \[H^i(H,\Z/p\Z)\xrightarrow{\on{Res}} H^i(H',\Z/p\Z)\xrightarrow{\on{Cor}}H^i(H,\Z/p\Z)\]
 are isomorphisms. Therefore the sequence of \Cref{cyclicext-module} induced by $\chi$ is a direct summand of the sequence induced by $\chi'$. In particular, if the latter is exact, so is the former. Replacing $(G,\theta)$ by $(G',\theta|_{G'})$, we may thus assume that $(G,\theta)$ is cyclotomic.
  
  The conclusion is clear if $\chi=0$, hence we may assume that $\chi\neq 0$. We first prove exactness at $H^2(G,\Z/p\Z)$. Let $C\coloneqq G/H$, so that $C$ is a cyclic group of order $p$, and let $\cl{\chi}: C \to \mathbb{Z}/p\mathbb{Z}$ be the homomorphism induced by $\chi$. Let $M$ be a Hilbert 90 module for $(G,\theta)$, fix an inclusion $\iota\colon \Z/p\Z\hookrightarrow M$, and write $\iota_*\colon H^2(G,\Z/p\Z)\to H^2(G,M)$ for the induced map.

  Let $\alpha\in H^2(G,\Z/p\Z)$ be such that $\on{Res}^G_H(\alpha)=0$, and set $\alpha_1\coloneqq \iota_*(\alpha)$. Since $H^1(H, M)=0$, by \cite[Proposition (1.6.7)]{neukirch2008cohomology} the sequence
		\[
		0\to H^2(C,M^H)\xrightarrow{\on{Inf}} H^2(G,M) \xrightarrow{\on{Res}} H^2(H,M)
		\]
		is exact. As $\alpha$ restricts to zero on $H$, so does $\alpha_1$, and so there is $\alpha_2\in H^2(C,M^H)$ such that $\alpha_1=\on{Inf}^G_C(\alpha_2)$. Let \[\beta\colon H^1(G,\Z/p\Z)\to H^2(G,\Z),\qquad \cl{\beta}\colon H^1(C,\Z/p\Z)\to H^2(C,\Z),\qquad\] be the connecting homomorphisms associated to the short exact sequence \[0\to \Z\xrightarrow{p}\Z\to\Z/p\Z\to 0,\]
  viewed as a sequence of trivial $G$-modules and trivial $C$-modules, respectively.
  Since cup products commute with inflation maps (see \cite[Proposition (1.5.3)(ii)]{neukirch2008cohomology}), we have a commutative square
\[
    \begin{tikzcd}
        H^0(C,M^H) \arrow[r,"\cup\cl{\beta}(\cl{\chi})"] \arrow[d,equal,"\on{Inf}"] & H^2(C,M^H) \arrow[d,"\on{Inf}"] \\
        H^0(G,M) \arrow[r,"\cup\beta(\chi)"] & H^2(G,M). 
    \end{tikzcd}
\]
Note that $H^0(C,M^H)=H^0(G,M)=M^G$. As the group $C$ is cyclic, the top horizontal map is surjective; see \cite[Proposition (1.7.1)]{neukirch2008cohomology}. Thus, there exists $m\in M^G$ such that $m\cup\cl{\beta}(\cl{\chi})=\alpha_2$ in $H^2(C,M^H)$, and hence $m\cup\beta(\chi)=\alpha_1$ in $H^2(G,M)$. 

  By \cite[Proposition 3.4.9]{gille2017central} 
  applied to the pairing of $G$-modules $M\times \Z\to M$ given by $(v,n)\mapsto nv$, the diagram
	\[\begin{tikzcd}    H^1(G,M[p]) \otimes  H^1(G,\mathbb{Z}/p\mathbb{Z}) \arrow[d,"\beta", {xshift=5ex}] \arrow[r,"\cup"] & H^2(G,\Z/p\Z) \arrow[d,"\iota_*"] \\
		H^0(G,M) \arrow[u,"\partial",{xshift=-5ex}] \otimes H^2(G,\mathbb{Z}) \arrow[r,"\cup"] & H^2(G,M),
	\end{tikzcd}
	\]
  is anti-commutative, that is, $\iota_*(\partial(v)\cup \psi) =- v\cup\beta(\psi)$ for all $v\in H^0(G,M)=M^G$ and $\psi\in H^1(G,\Z/p\Z)$. Therefore 
  \[
  \iota_*(\alpha)=\alpha_1=m\cup\beta(\chi)=-\iota_*(\partial(m)\cup \chi)=\iota_*(-\partial(m)\cup\chi).
  \]
Since $H^1(G,M)=0$, the map $\iota_*$ is injective, and hence $\alpha=-\partial(m)\cup\chi$. This proves exactness at $H^2(G,\Z/p\Z)$. Finally, by \cite[Proposition 15.5]{merkurjev1982k-cohomology}, the sequence is also exact at $H^1(G,\Z/p\Z)$.
	\end{proof}

    \begin{cor}\label{cup-norm}
        Let $(G,\theta)$ be a $p$-oriented profinite group and let $M$ be a Hilbert 90 module for $(G,\theta)$. Let $a,b\in M^G$ be such that $\chi_a\cup\chi_b=0$ in $H^2(G,\Z/p\Z)$. Then there exists $\alpha\in M^{G_a}$ such that $N_{G/G_a}(\alpha)=b$ in $M^G$.
    \end{cor}

    \begin{proof}
    By \Cref{cyclicext-module} applied to $\chi=\chi_a$, there exists $\psi\in H^1(G_a,\Z/p\Z)$ such that $\on{Cor}^G_{G_a}(\psi)=\chi_b$. Let $\alpha'\in M^{G_a}$ be such that $\psi=\chi_{\alpha'}$. By \Cref{norm-cor}, we have $N_{G/G_a}(\alpha')=b+pu$ for some $u\in M^G$. Define $\alpha\coloneqq \alpha'-u\in M^{G_a}$. Then $N_{G/G_a}(\alpha)=b$, as desired.
    \end{proof}

	\section{Proof of Theorem \ref{hilbert-90-equiv}}\label{section4}
	
	\begin{proof}[Proof of \Cref{hilbert-90-equiv}]
    Suppose first that $(G,\theta)$ admits a Hilbert 90 module $M$. We have a commutative diagram of $G$-modules with exact rows
		\[\begin{tikzcd}    0 \arrow[r] & S[p^n] \arrow[d] \arrow[r] & M \arrow[d,"p^{n-1}"] \arrow[r,"p^{n}"] & M \arrow[d,equals] \arrow[r] & 0 \\
			0 \arrow[r] & S[p] \arrow[r] & M \arrow[r,"p"] & M \arrow[r] & 0.
		\end{tikzcd}
		\]
		Therefore, for every open subgroup $H$ of $G$, we obtain a commutative diagram with exact rows
		\[
		\begin{tikzcd}
			M^H \arrow[d,equals] \arrow[r] & H^1(H,S[p^n]) \arrow[d] \arrow[r] & H^1(H,M) \arrow[d,"p^{n-1}"] \\
			M^H \arrow[r] & H^1(H,S[p]) \arrow[r] & H^1(H,M).
		\end{tikzcd}
		\]
		Since $H^1(H,M) = 0$, we deduce that the map $H^1(H,S[p^n])\to H^1(H,S[p])$ is surjective, as desired.

\medskip
    
	Conversely, suppose that the pair $(G,\theta)$ satisfies formal Hilbert 90. 
 Let $H$ be an open subgroup of $G$, and fix $\sigma\in Z^1(H, S[p])$.
		By \Cref{surjective-induction}, we may find a sequence $(\sigma_n)_{n\geq 1}$, where $\sigma_n\in Z^1(H,S[p^n])$, $\sigma_1=\sigma$ and the map $Z^1(H,S[p^n])\to Z^1(H,S[p^{n-1}])$ sends $\sigma_n$ to $\sigma_{n-1}$ for all $n\geq 2$.
  
  \begin{lemma}\label{mult-by-p}
       Write $\cl{\sigma}_{n}$ for the image of $\sigma_{n}$ under $Z^1(H,S[p^n])\to Z^1(H,S)$. For all $n\geq 1$ and all $x\in \Z_{(p)}$, we have $x\cl{\sigma}_n=px\cl{\sigma}_{n+1}$ in $Z^1(H,S)$.
  \end{lemma}

  \begin{proof}
  We have a commutative square $H$-modules
      \[
    \begin{tikzcd}
        S[p^n]  \arrow[r,hook] \arrow[d,"p"] & S \arrow[d,"p"]\\
        S[p^{n-1}] \arrow[r,hook] & S.
    \end{tikzcd}
  \]
  Passing to cocycles, we get a commutative square
  \[
    \begin{tikzcd}
        Z^1(H,S[p^n]) \arrow[r] \arrow[d] & Z^1(H,S) \arrow[d,"p"]  \\
        Z^1(H,S[p^{n-1}]) \arrow[r] & Z^1(H,S),
    \end{tikzcd}
  \]
 from which the conclusion follows.
  \end{proof}
		
		Consider the uniquely divisible group $Q_H\coloneqq \oplus_{\sigma} \Q u_{\sigma}$, where $\sigma$ ranges over all elements of $Z^1(H,S[p])$ and the $u_{\sigma}$ are formal symbols. 
        
        We now define an $H$-module $M_H$. As an abelian group, $M_H\coloneqq S\oplus Q_H$. The group $H$ acts on $M_H$ as follows: for all $h\in H$, $\sigma\in Z^1(H,S[p])$, $s\in S$, $a\in \Q$, if $n$ is a positive integer such that $p^n a\in\Z_{(p)}$, then
        \begin{equation}\label{formula-general}
			h(s,au_{\sigma})\coloneqq(\theta(h)s+p^n a\cl{\sigma}_n(h), au_{\sigma}).
		\end{equation}
    By \Cref{mult-by-p}, this definition is independent on the choice of $n$. We have a short exact sequence of $H$-modules
        \begin{equation}\label{m-h-exact}0\to S\to M_H\to Q_H\to 0,\end{equation}
    where the map $S\to M_H$ is the inclusion into the first factor and the map $M_H\to Q_H$ is the projection onto the second factor. 

    \begin{lemma}\label{m-h-lemma}
    The map $H^1(H,S[p])=H^1(H,M_H[p])\to H^1(H,M_H)$ is trivial.
    \end{lemma}
		  
		\begin{proof}
        Let $\sigma\in Z^1(H,S[p])$ be a $1$-cocycle, and set $m_{\sigma}\coloneqq (0, \frac{1}{p}u_{\sigma})\in M_H$.
		By (\ref{formula-general}), for all $h\in H$ we have $(h-1)m_{\sigma}=(\cl{\sigma}_1(h),0)=(\cl{\sigma}(h),0)$, that is, the image of $\sigma$ in $Z^1(H,M_H)$ is equal to the coboundary of $m_{\sigma}$.  
		\end{proof}
  Consider the following diagram of short exact sequence of $G$-modules
    \begin{equation}\label{m-h-tilde}
    \begin{tikzcd}
    0\arrow[r] & \on{Ind}_H^G(S) \arrow[r] \arrow[d,"\nu"] & \on{Ind}_H^G(M_H) \arrow[r] \arrow[d]  & \on{Ind}_H^G(Q_H) \arrow[r]\arrow[d,equal]  & 0\\
    0\arrow[r] & S \arrow[r] & \widetilde{M}_H \arrow[r] & \on{Ind}_H^G(Q_H) \arrow[r] & 0,
    \end{tikzcd}
    \end{equation}
    where the top row is induced by (\ref{m-h-exact}), $\nu$ is the map of (\ref{ind-norm}) for $A=S$, and $\widetilde{M}_H$ is defined as the pushout of the left square.
		If $N$ is an $H$-module, the induced $G$-module $\on{Ind}_H^G(N)$ is a $G$-module whose underlying abelian group is a direct sum of $[G:H]$ copies of $N$. In particular, the functor $\on{Ind}_H^G(-)$ is exact. This explains the exactness of the top row in (\ref{m-h-tilde}).

  \begin{lemma}\label{m-h-tilde-lemma}
  The map $H^1(H,S[p])=H^1(H,\widetilde M_H[p])\to H^1(H,\widetilde M_H)$ is trivial.
  \end{lemma}

  \begin{proof}
    This follows from the fact that the inclusion of $S[p]=\widetilde M_H[p]$ into $\widetilde M_H$ factors as the composition of the $H$-module homomorphisms
		\[
		S[p]\to M_H\to \on{Ind}_H^G(M_H)\to \widetilde M_H
		\]
		and the homomorphism $H^1(H,S[p])\to H^1(H,M_H)$ is trivial by \Cref{m-h-lemma}.
  \end{proof}

Consider the following diagram of $G$-modules with exact rows
\begin{equation}\label{m}
    \begin{tikzcd}
    0\arrow[r] &  \bigoplus_{H\subset G}S  \arrow[r] \arrow[d] & \bigoplus_{H\subset G} \widetilde{M}_H \arrow[r] \arrow[d]  & \bigoplus_{H\subset G}\on{Ind}_H^G(Q_H) \arrow[r]\arrow[d,equal]  & 0\\
    0\arrow[r] & S \arrow[r] & M \arrow[r] & \bigoplus_{H\subset G}\on{Ind}_H^G(Q_H) \arrow[r] & 0.
    \end{tikzcd}
    \end{equation}
    Here $H$ ranges over all open subgroups of $G$, the top row is the direct sum of the bottom rows in (\ref{m-h-tilde}), the left vertical map is the identity on every summand, and $M$ is defined so that the left square is a pushout.

    \begin{lemma}\label{explicit-h-90-module}
    The $G$-module $M$ is a Hilbert 90 module for $(G,\theta)$.
    \end{lemma}

    \begin{proof}
     By (\ref{m}), the $G$-module $M$ is an extension of a uniquely divisible $G$-module by the $p$-primary torsion $G$-module $S$. It follows that $M$ is divisible and uniquely divisible by all primes $p'\neq p$. In particular, $M=pM$ and $M_{\on{tors}}=S$.
        
    Let $H\subset G$ be an open subgroup. By (\ref{m}), the embedding of $S$ into $M$ factors as $S\to \widetilde M_H\to M$. Therefore, by \Cref{m-h-tilde-lemma}, the homomorphism \[H^1(H,S[p])=H^1(H,M[p])\to H^1(H,M)\] is trivial. The short exact sequence \[0\to S[p]\to M\xrightarrow{p} M \to 0\] yields an exact sequence
		\[
		H^1(H,S[p])\to H^1(H,M) \xrightarrow{p} H^1(H,M).
		\]
		Since the first homomorphism is trivial, this implies that $H^1(H,M)[p]=0$. For all primes $p'\neq p$,  the subgroup $H^1(H,M)[p']$ is trivial because $M$ is uniquely $p'$-divisible. We conclude that $H^1(H,M)=0$. This completes the proof.
    \end{proof}
    By \Cref{explicit-h-90-module}, the $p$-oriented profinite group $(G,\theta)$ admits a Hilbert 90 module. This completes the proof of \Cref{hilbert-90-equiv}.
    \end{proof}

	\begin{cor}\label{cyclicext}
		Let $(G,\theta)$ be a $p$-oriented profinite group that satisfies formal Hilbert 90. Let $\chi:G\to\Z/p\Z$ be a character and set $H\coloneqq\on{Ker}(\chi)$.
		Then the sequence
		\[
		H^1(H,\Z/p\Z)\xrightarrow{\on{Cor}} H^1(G,\Z/p\Z) \xrightarrow{\cup\chi} H^2(G,\Z/p\Z)\xrightarrow{\on{Res}} H^2(H,\Z/p\Z)
		\]
		is exact.
	\end{cor} 

 \begin{proof}
     Combine \Cref{hilbert-90-equiv} and \Cref{cyclicext-module}.
 \end{proof}

\begin{rmk}\label{always-exact}
    When $p=2$ the sequence of \Cref{cyclicext} is exact for all profinite groups $G$, as it is part of the long exact sequence associated to
    \[0\to \Z/2\Z \to \on{Ind}_H^G(\Z/2\Z)\to \Z/2\Z\to 0.\]
\end{rmk}

    \section{Massey products in group cohomology}\label{massey-section}
\subsection{Massey products} Let $(A,\partial)$ be a differential graded ring, i.e, $A=\oplus_{i\geq 0}A^i$ is a non-negatively graded abelian group with an associative multiplication which respects the grading, and $\partial\colon A\to A$ is a group homomorphism of degree $1$ such that $\partial\circ \partial=0$ and $\partial(ab)=\partial(a)b+(-1)^ia\partial(b)$ for all $i\geq 0$, $a\in A^i$ and $b\in A$. We denote by $H^*(A)\coloneqq \on{Ker}\partial/\on{Im}\partial$ the cohomology of $(A,\partial)$, and we write $\cup$ for the multiplication (cup product) on $H^*(A)$.
	
	Let $n\geq 2$ be an integer and $a_1,\dots,a_n\in H^1(A)$. A \emph{defining system} for the $n$-th order Massey product $\ang{a_1,\dots,a_n}$ is a collection $M$ of elements of $a_{ij}\in A^1$, where $1\leq i<j\leq n+1$, $(i,j)\neq (1,n+1)$, such that
 \begin{enumerate}
     \item $\partial(a_{i,i+1})=0$ and $a_{i,i+1}$ represents $a_i$ in $H^1(A)$, and
     \item $\partial(a_{ij})=-\sum_{l=i+1}^{j-1}a_{il}a_{lj}$ for all $i<j-1$.
 \end{enumerate}
	
	It follows from (2) that  $-\sum_{l=2}^na_{1l}a_{l,n+1}$ is a $2$-cocycle: we write $\ang{a_1,\dots,a_n}_M$ for its cohomology class in $H^2(A)$, called the \emph{value} of $\ang{a_1,\dots,a_n}$ corresponding to $M$. By definition, the \emph{Massey product} of $a_1,\dots,a_n$ is the subset $\ang{a_1,\dots,a_n}$ of $H^2(A)$ consisting of the values $\ang{a_1,\dots,a_n}_M$ of all defining systems $M$. We say that the Massey product $\ang{a_1,\dots,a_n}$ is \emph{defined} if it is non-empty, and that it \emph{vanishes} if $0\in \ang{a_1,\dots,a_n}$.

    \begin{lemma}\label{trivial-case}
        Let $A$ be a differential graded ring and let $a_1,\dots,a_n\in H^1(A)$. Suppose that the Massey product $\ang{a_1,\dots,a_n}$ is defined and that $a_k=0$ for at least one $k\in\set{1,\dots,n}$. Then the Massey product $\ang{a_1,\dots,a_n}$ vanishes.
    \end{lemma}

    \begin{proof}
        We follow the proof \cite[Lemma 6.2.4]{fenn1983techniques}. Let $(a_{ij})_{i,j}$ be a defining system for $\ang{a_1,\dots,a_n}$. Set $b_{ij}=a_{ij}$ if $i>k$ or $j<k$, and set $b_{ij}=0$ if $i\leq k\leq j$. Then $(b_{ij})_{i,j}$ is a defining system for $\ang{a_1,\dots,a_n}$. Moreover, for all $2\leq l\leq n$ we have $b_{1l}b_{l,n+1}=0$, and hence $\sum_{l=2}^nb_{1l}b_{l,n+1}=0$. Thus $\ang{a_1,\dots,a_n}$ vanishes.
    \end{proof}

	\subsection{Dwyer's Theorem}
 
	Let $G$ be a profinite group and let $p$ be a prime number. The complex $(C^*(G,\Z/p\Z),\partial)$ of mod $p$ non-homogeneous continuous cochains of $G$ with the standard cup product is a differential graded ring. Thus the cohomology ring $H^*(G,\Z/p\Z)=H^*(C^*(G,\Z/p\Z),\partial)$ is endowed with Massey products.
	  
 Let $U_{n+1}\subset \on{GL}_{n+1}(\F_p)$ be the subgroup of all upper-triangular unipotent matrices. For all $1\leq i<j\leq n+1$, we denote by $u_{i,j}\colon U_{n+1}\to \Z/p\Z$ the $(i,j)$-th coordinate function on $U_{n+1}$. The functions $u_{i,j}$ are group homomorphisms only when $j=i+1$. The center $Z_{n+1}$ of $U_{n+1}$ is the subgroup of all matrices such that $u_{i,j}=0$ when $(i,j)\neq (1,n+1)$. We fix an identification $Z_{n+1}\simeq \Z/p\Z$ by letting $1$ correspond to the matrix in $Z_{n+1}$ with $1$ on the top-right corner. We obtain a commutative diagram
	\begin{equation}\label{central-ses}
	\begin{tikzcd}
		1\arrow[r] & \Z/p\Z\arrow[r] &  U_{n+1}\arrow[r]  \arrow[dr]& \cl{U}_{n+1}\arrow[r] \arrow[d]  & 1\\
		&&& (\Z/p\Z)^n
	\end{tikzcd}	
	\end{equation}
	where the row is a short exact sequence and the homomorphism $U_{n+1}\to (\Z/p\Z)^n$ is given by $(u_{1,2},u_{2,3},\dots, u_{n,n+1})$. The group $\cl{U}_{n+1}$ may be identified with the group of all upper triangular unipotent matrices of size $(n+1)\times (n+1)$ with the entry $(1,n+1)$ omitted.  
	
 The following characterization of Massey products in group cohomology is due to Dwyer \cite{dwyer1975homology}.

 \begin{thm}\label{dwyer}
  Let $p$ be a prime number, let $G$ be a profinite group, let $\chi_1,\dots,\chi_n$ be in $H^1(G,\Z/p\Z)$, write $\chi= (\chi_1,\dots,\chi_n)\colon G\to (\Z/p\Z)^n$ for the induced homomorphism, and consider the mod $p$ Massey product $\ang{\chi_1,\dots,\chi_n}$.
  
  (i) If $\tilde{\chi} : G\to \cl{U}_{n+1}$ is a continuous homomorphism lifting $\chi$, then the collection of 1-cochains $M_{\tilde{\chi}}\coloneqq \set{u_{i,j}\circ \tilde{\chi}}$ forms a defining system for $\ang{\chi_1,\dots,\chi_n}$.
    
(ii) There is a bijective correspondence between continuous lifts  of $\chi$ and defining systems for $\ang{\chi_1,\dots,\chi_n}$, given by sending a lift $\tilde{\chi} : G\to \cl{U}_{n+1}$ to $M_{\tilde{\chi}}$.
     
(iii) Let $\widetilde{\chi}\colon G\to \cl{U}_{n+1}$ be a continuous lift of $\chi$. Then the value of $M_{\widetilde{\chi}}$ is equal to $-\widetilde{\chi}^*(u)\in H^2(G,\Z/p\Z)$, where $u\in H^2(\cl{U}_{n+1},\Z/p\Z)$ is the class of the short exact sequence of (\ref{central-ses})
  
(iv) In particular, $\ang{\chi_1,\dots,\chi_n}$ is defined (resp. vanishes) if and only if $\chi$ lifts to a continous homomorphism $G\to \cl{U}_{n+1}$ (resp. $G\to U_{n+1}$). 
	\end{thm}

 It follows from (iii) that $\widetilde{\chi}$ lifts to $U_{n+1}$ lifts to $U_{n+1}$ if and only if the class
 \[h^*(u)=\left[\sum_{l=2}^{n}\widetilde{\chi}_{1l}\widetilde{\chi}_{l,n+1}\right]\in H^2(G,\Z/p\Z),\]
 vanishes, where $\widetilde{\chi}_{ij}\coloneqq u_{ij}\circ \widetilde{\chi}$. We will refer to $h^*(u)$ as the \emph{obstruction} to lifting $\widetilde{\chi}$ to $U_{n+1}$.
	
	\begin{proof}
	This result is originally due to Dwyer \cite{dwyer1975homology}, with different sign conventions. For the version of Dwyer's Theorem with our sign conventions, see \cite[Proposition 8.3]{efrat2014zassenhaus} or  \cite[Proposition 2.2]{harpaz2019massey}.
	\end{proof}

Let $C_p$ be a cyclic group of order $p$, and fix a generator $\sigma\in C_p$. For every $C_p$-module $A$, we define a $C_p$-module endomorphism
\[\widetilde{N}_{C_p}\colon A\to A,\qquad \widetilde{N}_{C_p}(a)=\sum_{l=0}^{p-2}(p-1-l)\sigma^la\qquad \text{for all $a\in A$.}\]
Now let $G$ be a profinite group, let $\chi\in H^1(G,\Z/p\Z)$ be a non-zero character, let $H$ be the kernel of $\chi$, and let $g\in G$ be such that $\chi(g)=1$. For every $G$-module $B$ on which $H$ acts trivially, we define a $G$-module endomorphism $\widetilde{N}_{\chi}\colon B\to B$ as $\widetilde{N}_{C_p}$, where we view $B$ as a $G/H$-module and identify $G/H$ with $C_p$ so that $gH$ maps to $\sigma$. Explicitly:
\[\widetilde{N}_{\chi}\colon B\to B,\qquad \widetilde{N}_{\chi}(b)=\sum_{l=0}^{p-2}(p-1-l)g^lb\qquad \text{for all $b\in B$.}\]
A simple calculation shows that
\begin{equation}\label{g-1}(g-1)\widetilde{N}_{\chi}(b)=N_{G/H}(b)-pb\qquad\text{for all $b\in B$.}\end{equation}
For all $i\geq 0$, we have the induced maps $\widetilde{N}_{\chi}\colon H^i(H,B)\to H^i(H,B)$.

\begin{lemma}\label{p-u3}
   Let $p$ be a prime, let $G$ be a profinite group, and let $\chi_1,\chi_2\in H^1(G,\Z/p\Z)$. Suppose that $\chi_1\neq 0$, and set $H\coloneqq \on{Ker}(\chi_1)$.
   
   (1) Let $\psi\in H^1(H,\Z/p\Z)$ be such that $\on{Cor}^G_{H}(\psi)=\chi_2$, and define \[\eta\coloneqq \widetilde{N}_{\chi_1}(\psi)\in H^1(H,\Z/p\Z).\] Then there exists a continuous homomorphism $G\to U_3$ given by a matrix 
\begin{equation}\label{a-b-rho}
	\begin{bmatrix}
		1 & \chi_1 & \rho\\
		0 & 1 & \chi_2 \\
		0 & 0 & 1
	\end{bmatrix}
\end{equation}
such that $\rho|_{H}=\eta$.

(2) Suppose that $p=2$. For all $\psi\in H^1(H,\Z/2\Z)$ such that $\on{Cor}^G_{H}(\psi)=\chi_2$, there exists a continuous homomorphism $G\to U_3$ as in (\ref{a-b-rho}) such that $\rho|_{H}=\psi$. Conversely, given a continuous homomorphism $G\to U_3$ as in (\ref{a-b-rho}) and setting $\psi\coloneqq \rho|_{H}$, we have $\on{Cor}^G_{H}(\psi)=\chi_2$. 
\end{lemma}
	\begin{proof}
    (1) We have $U_3=V\rtimes C$, where
    \[
        V=\begin{bmatrix}
			1 & 0 & *  \\
			& 1 & *  \\
			&    & 1
		\end{bmatrix},\qquad 
		 C=\begin{bmatrix}
			1 & * & 0  \\
			& 1 & 0  \\
			&    & 1
		\end{bmatrix}.
  \]
     We let $C$ act on $V$ by conjugation. We identify $C$ and $\Z/p\Z$ via the coordinate map $u_{1,2}$, and we write $V_1$ for the twist of $V$ by $\chi_1\colon G\to C=\Z/p\Z$. 
    
    Let $g\in G$ be such that $\chi_1(g)=1$. We define $j\colon V_1\to \on{Ind}^G_H(\Z/p\Z)$ by sending
    \[\begin{bmatrix}
        1 & 0 & 1 \\
        0 & 1 & 0 \\
        0 & 0 & 1
    \end{bmatrix},\qquad 
    \begin{bmatrix}
        1 & 0 & 0 \\
        0 & 1 & 1 \\
        0 & 0 & 1
    \end{bmatrix}\]
    to the constant function $1$ and the function mapping $g^iH$ to $i-1$ for $i=0,\dots,p-1$, respectively. We have the following commutative diagram:
\begin{equation}\label{alpha-theta-diagram}
\begin{tikzcd}
    \on{Ind}_{H}^G(\Z/p\Z) \arrow[r,->>,"\mu"] \arrow[rr,bend left =15,"\widetilde{N}_{\chi_1}"]  \arrow[d,equal] & V_1 \arrow[d, "\pi"] \arrow[r,hook,"j"] & \on{Ind}_{H}^G(\Z/p\Z) \arrow[d,"g-1"] \\  
     \on{Ind}_{H}^G(\Z/p\Z) \arrow[r,->>,"\nu"]\arrow[rr,bend right =15,"N_{G/H}",swap] & \Z/p\Z\arrow[r,hook,"i"]  & \on{Ind}_{H}^G(\Z/p\Z),
\end{tikzcd}
\end{equation}
where $\mu$ is defined by the commutativity of the top triangle, and $i$ and $\nu$ are the maps of (\ref{ind-inclusion}) and (\ref{ind-norm}) for $A=\Z/p\Z$. Indeed, the bottom triangle commutes by a direct calculation, the outer rectangle commutes by (\ref{g-1}), the right square commutes by a direct calculation, and by the injectivity of the right horizontal maps this implies the commutativity of the left square.

    Let $q\colon \on{Ind}_H^G(\Z/p\Z)\to \Z/p\Z$ be given by evaluation at the identity of $G$: it is a homomorphism of trivial $H$-modules. The Faddeev--Shapiro isomorphism
    \[\varphi\colon H^1(G,\on{Ind}_H^G(\Z/p\Z))\xrightarrow{\on{Res}} H^1(H,\on{Ind}_H^G(\Z/p\Z))\xrightarrow{q_*} H^1(H,\Z/p\Z)\]
    fits into a commutative square
    \begin{equation}\label{norm-shapiro-cor}
    \begin{tikzcd}
        H^1(G,\on{Ind}_H^G(\Z/p\Z)) \arrow[r,"\varphi"]\arrow[d,"\nu_*"]  & H^1(H,\Z/p\Z)\arrow[d,"\on{Cor}^G_H"] \\
        H^1(G,\Z/p\Z) \arrow[r,equal] & H^1(G,\Z/p\Z);
    \end{tikzcd}
    \end{equation}
    see \cite[Proposition (1.6.5)]{neukirch2008cohomology}.
    Let $\psi_0\coloneqq \varphi^{-1}(\psi)$ and $\eta_0\coloneqq \varphi^{-1}(\eta)$, so that $\widetilde{N}_{\chi_1}(\psi_0)=\eta_0$. 
    The homomorphism $H^1(G,\on{Ind}_{H}^G(\Z/p\Z))\to H^1(G,V_1)$ coming from (\ref{alpha-theta-diagram})
    sends $\psi_0$ to a class
    \[h_0\coloneqq \begin{bmatrix}
      1 & 0 &  \rho \\
      0 & 1 &  \chi \\
      0 & 0 & 1
    \end{bmatrix}\in H^1(G,V_1).
    \]
    The commutativity of the right square of (\ref{alpha-theta-diagram}) implies that $\chi=\pi_*(h_0)=\nu_*(\eta_0)$, and by (\ref{norm-shapiro-cor}) we have $\nu_*(\psi_0)=\on{Cor}^G_H(\psi)=\chi_2$. Therefore $\chi=\chi_2$. 
We have a commutative square of $H$-modules with trivial action
    \begin{equation}\label{used-for-shapiro}
    \begin{tikzcd}
    V_1 \arrow[d,"u_{1,3}"] \arrow[r, hook,"j"] & \on{Ind}_{H}^G(\Z/p\Z) \arrow[d,"q"]  \\
    \Z/p\Z \arrow[r, equal] & \Z/p\Z, 
    \end{tikzcd}
    \end{equation}
    where the top horizontal map comes from (\ref{alpha-theta-diagram}). We obtain the commutative square
    \begin{equation}\label{double-vertical}
    \begin{tikzcd}
        H^1(G,V_1) \arrow[r,"j_*"]\arrow[d,"\on{Res}"]  & H^1(G,\on{Ind}_{H}^G(\Z/p\Z)) \arrow[d,"\on{Res}"] \arrow[dd, bend left=75,"\varphi"] \\
        H^1(H,V_1) \arrow[r,"j_*"] \arrow[d,"(u_{1,3})_*"] & H^1(H,\on{Ind}_{H}^G(\Z/p\Z)) \arrow[d,"q_*"] \\
        H^1(H,\Z/p\Z) \arrow[r,equal] & H^1(H,\Z/p\Z),
    \end{tikzcd}
    \end{equation}
    where the bottom square is induced by (\ref{used-for-shapiro}). We know from (\ref{alpha-theta-diagram}) that the map $j_*\colon H^1(G,V_1)\to H^1(G,\on{Ind}_{H}^G(\Z/p\Z))$ sends $h_0$  
    to $\widetilde{N}_{\chi_1}(\psi_0)=\eta_0$. On the other hand, $\varphi(\eta_0)=\eta$ and the composition of the left vertical maps sends $h_0$ to $\rho|_{H}$. Thus $\rho|_{H}=\eta$.

    By \cite[Proposition (28.8)]{knus1998book}, the twisting map $Z^1(G,V_1)\to Z^1(G,U_3)$ sends $h_0$ to
    \[ 
    \begin{bmatrix}
			1 & 0 & \rho\\
			0 & 1 & \chi_2 \\
			0 & 0 & 1
		\end{bmatrix}\cdot
  \begin{bmatrix}
			1 & \chi_1 & 0\\
			0 & 1 & 0 \\
			0 & 0 & 1
		\end{bmatrix}
    =
    \begin{bmatrix}
			1 & \chi_1 & \rho\\
			0 & 1 & \chi_2 \\
			0 & 0 & 1
		\end{bmatrix}.
  \]
    This matrix defines a continuous homomorphism $G\to U_3$ such that $\rho|_{H}=\eta$.  
    
    (2) Since $\widetilde{N}_{\chi_1}(\psi)=\psi$ when $p=2$, the first part of (2) follows from an application of (1) to $\eta=\psi$. We now prove the second part of (2). Let $g\in G-H$. The map $g-1\colon \on{Ind}^G_H(\Z/2\Z)\to\on{Ind}^G_H(\Z/2\Z)$ factors through $\nu\colon \on{Ind}^G_H(\Z/2\Z)\to \Z/2\Z$. Therefore the right square of (\ref{alpha-theta-diagram}) yields a commutative square
    \[
    \begin{tikzcd}
    V_1 \arrow[r,"j"]\arrow[d,"\pi"]  & \on{Ind}_{H}^G(\Z/2\Z)\arrow[d,"\nu"] \\
    \Z/2\Z \arrow[r,equal] & \Z/2\Z.  
    \end{tikzcd}
    \]
    We obtain a commutative diagram
    \[
    \begin{tikzcd}
    H^1(G,V_1) \arrow[r,"j_*"] \arrow[d,"\pi_*"]  & H^1(G,\on{Ind}_{H}^G(\Z/2\Z)) \arrow[r,"\varphi", "\sim"']  \arrow[d,"\nu_*"] & H^1(H,\Z/2\Z) \arrow[d,"\on{Cor}^G_{H}"]\\
        H^1(G,\Z/2\Z) \arrow[r,equal] & H^1(G,\Z/2\Z) \arrow[r,equal] & H^1(G,\Z/2\Z),
    \end{tikzcd}
    \]
    where the right square is (\ref{norm-shapiro-cor}). Consider the class
    \[
    h\coloneqq \begin{bmatrix}
        1 & 0 & \rho \\
        0 & 1 & \chi_2 \\
        0 & 0 & 1
    \end{bmatrix}\in H^1(G,V_1).
    \]
    We deduce that $\chi_2=\pi_*(h)=\on{Cor}^G_H(\varphi(j_*(h)))$. By (\ref{double-vertical}), we have $\varphi(j_*(h))=\psi$, and hence $\chi_2=\on{Cor}^G_{H}(\psi)$, as desired.
    \end{proof}

	\section{Proof of Theorem \ref{main-h90}}\label{section6}

	\begin{proof}[Proof of \Cref{main-h90}]
It is clear that (3) implies (2), and it follows from the definition of triple Massey products that (1) and (2) are equivalent. It thus suffices to show that (2) implies (3).

Define $G'\coloneqq \theta^{-1}(1+p\Z_p)$. Then  $G'$ is an open subgroup of $G$, the index $[G:G']$ divides $p-1$ and $(G',\theta|_{G'})$ is cyclotomic. Suppose that (2) implies (3) for all triple Massey products in $H^2(G',\Z/p\Z)$. Then, as $p$ does not divide $[G:G']$, by \cite[Proposition 4.14]{minac2016triple} we deduce that (2) implies for (3) for all triple Massey products in $H^2(G,\Z/p\Z)$. Replacing $(G,\theta)$ by $(G',\theta|_{G'})$, we may thus assume that $(G,\theta)$ is cyclotomic. Moreover, by \Cref{trivial-case} we may suppose that $\chi_1\neq 0$ and $\chi_3\neq 0$.

By \Cref{hilbert-90-equiv}, the pair $(G,\theta)$ admits a Hilbert 90 module $M$. Fix an inclusion $\iota\colon \Z/p\Z\hookrightarrow M$, and let $a,b,c\in M^G$ such that $\chi_1=\chi_a$, $\chi_2=\chi_b$ and $\chi_3=\chi_c$; see below (\ref{connecting}) for the notation. Since $\chi_a\cup\chi_b=0$, by \Cref{cup-norm} there exists $\alpha\in M^{G_a}$ such that $b=N_{G/G_a}(\alpha)$. 

By \Cref{p-u3}(1), letting $\eta\coloneqq \widetilde{N}_{\chi_a}(\alpha)\in M^{G_a}$, there exists a continuous homomorphism $G\to U_3$ given by a matrix
    \[
    \begin{bmatrix}
    1 & \chi_a & \rho \\
    0 & 1 & \chi_b \\
    0 & 0 & 1
    \end{bmatrix}
    \]
    such that $\rho|_{G_a}=\chi_\eta$. Similarly, since $\chi_b\cup\chi_c=0$ there exists $\gamma\in M^{G_c}$ such that $N_{G/G_c}(\gamma)=b$. Letting $\mu\coloneqq \widetilde{N}_{\chi_a}(\gamma)\in M^{G_a}$, there exists a continuous homomorphism $G\to U_3$ given by a matrix
    \[
    \begin{bmatrix}
    1 & \chi_b & \pi \\
    0 & 1 & \chi_c \\
    0 & 0 & 1
    \end{bmatrix}
    \] 
    such that $\pi|_{G_c}=\chi_\mu$. The matrix
	\[
	\begin{bmatrix}
		1 & \chi_a  & \rho & \square \\
		0 & 1 & \chi_b & \pi \\
		0 & 0 & 1 & \chi_c \\
		0 & 0 & 0 & 1
	\end{bmatrix}
	\]
	defines a continuous homomorphism $G\to \cl{U}_4$. The obstruction to lifting this homomorphism to $U_4$ is given by
	\[
	\varepsilon:=\chi_a\cup \pi + \rho \cup \chi_c\in H^2(G,\Z/p\Z).
	\]
    Suppose first that $\chi_a$ and $\chi_c$ are linearly dependent in $H^1(G,\Z/p\Z)$. In this case $\on{Res}^G_{G_a}(\varepsilon)=0$, and so by \Cref{cyclicext-module} we have $\varepsilon=\chi_a\cup\chi_u$ for some $u\in M$. Consider the continuous homomorphism $G\to \cl{U}_4$ defined by the matrix
    \[
	\begin{bmatrix}
		1 & \chi_a  & \rho & \square \\
		0 & 1 & \chi_b & \pi-\chi_u \\
		0 & 0 & 1 & \chi_c \\
		0 & 0 & 0 & 1
	\end{bmatrix}.
	\]
    The obstruction to lifting this homomorphism to $U_4$ is given by
    \[\chi_a\cup(\pi-\chi_u)+\rho\cup\chi_c=\varepsilon-\chi_a\cup\chi_u=0,\]
   and hence this homomorphism lifts to $U_4$. The conclusion follows from \Cref{dwyer}.
    
	Suppose now that $\chi_a$ and $\chi_c$ are linearly independent in $H^1(G,\Z/p\Z)$. Then $G_a\neq G_c$, and hence
 \[N_{{G_{a+c}}/{G_a\cap G_c}}(\alpha-\gamma)=N_{G/G_a}(\alpha)-N_{G/G_c}(\gamma)=b-b=0.\] 
Let $g\in G_{a+c}$ such that $\chi_a(g)=1$. By \Cref{subgroup}, $M$ is a Hilbert 90 module for $(G_a,\theta|_{G_a})$, and hence by \Cref{omega-lemma} we have an exact sequence
    \[M^{G_a\cap G_c}\xrightarrow{g-1} M^{G_a\cap G_c}\xrightarrow{N_{G/G_c}} M^{G_a}.\]
Therefore there exists $\omega\in M^{G_a\cap G_c}$ such that $(g-1)\omega=\alpha-\gamma$. Set
\[
	e\coloneqq\eta+N_{G/G_c}(\omega)\in M^{G_a}.
	\]
Then
\begin{align*}
    (g-1)e&=(g-1)\eta + (g-1)(N_{G/G_c}(\omega)) \\
    &=(N_{G/G_a}-p)(\alpha) + N_{G/G_c}((g-1)\omega) \\
    &=b-p\alpha+N_{G/G_c}(\alpha-\gamma)\\
    &=b-p\alpha+p\alpha-b\\
    &=0,
\end{align*}
which proves that $e\in M^G$. We have
	\[
	\on{Res}^G_{G_a}(\varepsilon)=\rho|_{G_a} \cup \chi_c=\chi_\eta \cup \chi_c =\chi_e \cup \chi_c\qquad\text{in $H^2(G_a,\Z/p\Z)$,}
	\]
	that is, $\on{Res}^G_{G_a}(\varepsilon-\chi_e \cup \chi_c)=0$ in $H^2(G_a,\Z/p\Z)$.
	By Proposition \ref{cyclicext} applied to $\chi=\chi_a$, there exists $f\in M^G$ such that $\varepsilon-\chi_e \cup \chi_c=\chi_a \cup \chi_f$ in $H^2(G,\Z/p\Z)$, and hence
	\[
	\varepsilon=\chi_a \cup \chi_f+\chi_e \cup \chi_c\qquad\text{in $H^2(G,\Z/p\Z).$}
	\]
	Consider the continuous homomorphism $G\to \cl{U}_4$ given by the matrix
    \[
	\begin{bmatrix}
		1 & \chi_a  & \rho-\chi_e & \square \\
		0 & 1 & \chi_b & \pi-\chi_f \\
		0 & 0 & 1 & \chi_c \\
		0 & 0 & 0 & 1
	\end{bmatrix}.
	\]
 The obstruction to lifting this homomorphism to $U_4$ is given by
 \[\chi_a\cup(\pi-\chi_f)+(\rho-\chi_e)\cup\chi_c=\varepsilon-\chi_a \cup \chi_f-\chi_e \cup \chi_c=0,\]
 and hence this homomorphism lifts to $U_4$. By \Cref{dwyer}, the Massey product $\ang{\chi_1,\chi_2,\chi_3}$ vanishes, as desired.
	\end{proof}

\begin{rmk}\label{comments1}
    When $p>2$, it is proved in \cite[Proposition 6.1.4, Theorem 6.2.1]{lam2023generalized} that if the sequence of \Cref{cyclicext} is exact at $H^1(G,\Z/p\Z)$ for a profinite group $G$, then the conclusion of \Cref{main-h90} holds for $G$. On the other hand, when $p=2$, the exactness of the sequence of \Cref{cyclicext} is satisfied by all profinite groups $G$ (see \Cref{always-exact}), and so it cannot imply the conclusion of \Cref{main-h90}. 
    
    We emphasize that, even when $p$ is odd, \Cref{main-h90} does not follow from \cite[Theorem 6.2.1]{lam2023generalized}. Indeed, the proof of \Cref{main-h90} relies on the construction of a Hilbert 90 module given by \Cref{hilbert-90-equiv}. Moreover, our proof of Theorems \ref{main-h90} and \ref{hilbert-90-equiv} is uniform in $p$ and applies to $p=2$ as well.
\end{rmk}

\begin{rmk}\label{comments2}
    \Cref{main-h90} proves the first half of a conjecture of Quadrelli \cite[Conjecture 1.5]{quadrelli2021chasing}, which predicts that \Cref{main-h90} might be true under the further assumption that $\on{Im}(\theta)\subset 1+4\Z_2$ when $p=2$.    The second half of \cite[Conjecture 1.5]{quadrelli2021chasing} predicts that, under the same assumptions, a fourfold Massey product $\ang{\chi_1,\chi_2,\chi_3,\chi_4}\subset H^2(G,\Z/p\Z)$ vanishes as soon as $\chi_1\cup\chi_2=\chi_2\cup\chi_3=\chi_3\cup\chi_4=0$. This is false in general; see \cite[Theorem 1.6]{merkurjev2022degenerate}.
\end{rmk}

\begin{rmk}
    Consider the group $G$ given by the following presentation:
    \[G\coloneqq \ang{x_1,x_2,x_3,x_4,x_5: [[x_1,x_2],x_3][x_4,x_5]=1}.\]
    Let $p$ be a prime number, and write $\hat{G}$ for the pro-$p$ completion of $G$. By \cite[Example 7.2]{minac2017triple}, there is a triple Massey product $\ang{\chi_1,\chi_2,\chi_3}\subset H^2(\hat{G},\Z/p\Z)$ which is defined but non-vanishing. By \Cref{main-h90}, we conclude that $(\hat{G},\theta)$ does not satisfy formal Hilbert 90 for any orientation $\theta$. When $p$ is odd, this fact already been proved by Quadrelli \cite[Theorem 1.5]{quadrelli2021chasing}.
\end{rmk}

\section{Proof of Theorem \ref{main-h90-degenerate}}\label{section7}

\begin{prop}\label{u5-criterion}
	Let $(G,\theta)$ be a $p$-oriented profinite group satisfying formal Hilbert $90$, and let $\chi_1,\chi_2,\chi_3\in H^1(G,\Z/2\Z)$. Suppose that $\chi_1\neq 0$ and let $H\coloneqq \on{Ker}(\chi_1)$.
 
	(1) The Massey product $\langle \chi_1,\chi_2,\chi_3,\chi_1\rangle$ vanishes if and only if there exist $\varphi$ and $\psi$ in $H^1(H,\Z/2\Z)$ such that $\on{Cor}^G_{H}(\varphi)=\chi_2$, $\on{Cor}^G_{H}(\psi)=\chi_3$, $\varphi\cup \psi=0$ and $\varphi\cup \on{Res}^G_{H}(\chi_3)=0$.
	
 (2) The Massey product $\langle \chi_1,\chi_2,\chi_3,\chi_1\rangle$ is defined if and only if there exist $\varphi$ and $\psi$ in $H^1(H,\Z/2\Z)$ such that $\on{Cor}^G_{H}(\varphi)=\chi_2$, $\on{Cor}^G_{H}(\psi)=\chi_3$,  $\varphi\cup \psi$ belongs to the image of $\on{Res}\colon H^2(G,\Z/2\Z)\to H^2(H,\Z/2\Z)$ and $\varphi\cup \on{Res}^G_{H}(\chi_3)=0$.
\end{prop}

\begin{proof}
Let $P$ be the normal subgroup of $U_5\subset \on{GL}_5(\F_2)$ given by 
	\[P\coloneqq \begin{bmatrix}
		1 & 0 & 0 & * & * \\
		& 1 & 0 & * & * \\
		&   & 1 & 0 & 0 \\
		&   &   & 1 & 0 \\
		&   &   &   & 1
	\end{bmatrix}.\]
	Note that $P$ is the kernel of the surjective homomorphism $U_5\to U_3\times U_3$ which forgets the $2\times 2$ upper-right square of a unitriangular matrix. We write $\cl{P}$ for the quotient of $P$ by the center of $U_5$. We obtain a commutative diagram
	\begin{equation}\label{cocycle-diag}
	\begin{tikzcd}
	& 1\arrow[d]  & 1 \arrow[d]\\
	& \Z/2\Z \arrow[r,equal] \arrow[d]  & \Z/2\Z \arrow[d] \\	
	1 \arrow[r] & P \arrow[d] \arrow[r] & U_5 \arrow[d] \arrow[r] & U_3\times U_3 \arrow[r] \arrow[d, equal] & 1 \\
	1 \arrow[r] & \cl{P}\arrow[d] \arrow[r] & \cl{U}_5\arrow[d] \arrow[r] & U_3\times U_3 \arrow[r] & 1\\
	& 1 & 1
	\end{tikzcd}
	\end{equation}
	where the rows and columns are short exact sequences. Since $P$ is abelian, the rows endow $P$ and $\cl{P}$ with the structure of $(U_3\times U_3)$-modules, and the quotient map $P\to \cl{P}$ is $(U_3\times U_3)$-equivariant.

    \begin{lemma}\label{u3u3}
Let $\varphi,\psi\in H^1(H,\Z/2\Z)$ be such that $\on{Cor}^G_{H}(\varphi)=\chi_2$ and $\on{Cor}^G_{H}(\psi)=\chi_3$, and let $h,h':G\to U_3$ be continuous homomorphisms given by matrices
 \begin{equation}\label{h-h'-matrices}
    \begin{bmatrix}
        1 & \chi_1 & \rho \\
        0 & 1 & \chi_2 \\
        0 & 0 & 1
    \end{bmatrix},\qquad 
    \begin{bmatrix}
        1 & \chi_3 & \rho' \\
        0 & 1 & \chi_1 \\
        0 & 0 & 1
    \end{bmatrix},
 \end{equation}
 where $\rho|_{H}=\varphi$ and $(\rho')|_{H}=\psi$.
    
(1) The homomorphism $(h,h')\colon G\to U_3\times U_3$ lifts to a continuous homomorphism $G\to U_5$ if and only if $\varphi\cup\psi=\varphi\cup \on{Res}^G_{H}(\chi_3)=0$.

(2) The homomorphism $(h,h')\colon G\to U_3\times U_3$ lifts to a continuous homomorphism $G\to \cl{U}_5$ if and only if $\varphi\cup\psi$ belongs to the image of $H^2(G,\Z/2\Z)\to H^2(H,\Z/2\Z)$ and $\varphi\cup \on{Res}^G_{H}(\chi_3)=0$.
    \end{lemma}
  
	\begin{proof}
 Consider the following normal subgroups of $U_3$:
\[
N\coloneqq \begin{bmatrix}
    1 & 0 & *\\
    & 1 & * \\
    & & 1
\end{bmatrix},\qquad
N'\coloneqq \begin{bmatrix}
    1 & * & *\\
    & 1 & 0 \\
    & & 1
\end{bmatrix}.
\]
We have an isomorphism of $(U_3\times U_3)$-modules $N\otimes N'\to P$ given by
\[
		\begin{bmatrix}
			1 & 0 & f_1  \\
			& 1 & e_1  \\
			&    & 1
		\end{bmatrix}\otimes\begin{bmatrix}
			1 & e_2 & f_2  \\
			& 1 & 0  \\
			&    & 1
		\end{bmatrix}\mapsto
  \begin{bmatrix}
			1 & 0 & 0 & f_1 e_2 & f_1 f_2 \\
			& 1 & 0 & e_1 e_2 & e_1 f_2 \\
			&   & 1 & 0 & 0 \\
            &   &   & 1 & 0 \\
            &   &   &   & 1
		\end{bmatrix}.
		\]

 We obtain a homomorphism 
 \[H^1(U_3,N)\otimes H^1(U_3,N')\xrightarrow{\cup} H^2(G,P).\]
 The natural projections $t:U_3\to N$ and $t':U_3\to N'$ are  $1$-cocycles, and a direct calculation shows that the class of the middle row of (\ref{cocycle-diag}) in $H^2(U_3\times U_3,P)$ is equal to $t\cup t'$. Consider the commutative diagram
\[
		\begin{tikzcd}
			H^1(U_3,N)\otimes H^1(U_3,N') \arrow[d,"h^*\otimes (h')^*"] \arrow[r,"\cup"] & H^2(U_3 \times U_3,P) \arrow[d,"\text{$(h,h')^*$}"] \\
			H^1(G,N)\otimes H^1(G,N')  \arrow[d,"\on{Res}\otimes\on{Res}"] \arrow[r,"\cup"] & H^2(G,P) \arrow[d,"\on{Res}"] \\	
            H^1(H,N)\otimes H^1(H,N')  \arrow[r,"\cup"] & H^2(H,P).
		\end{tikzcd}
		\]
It follows that the map $H^1(U_3\times U_3,P)\to H^2(H,P)$ sends
\begin{equation}\label{restrict-obstruction}
  t\cup t'\mapsto \begin{bmatrix}
			1 & 0 & 0 & \varphi\cup \on{Res}^G_H(\chi_3) & \varphi\cup\psi \\
			& 1 & 0 & \on{Res}^G_H(\chi_2\cup\chi_3) & \on{Res}^G_H(\chi_2)\cup\psi \\
			&   & 1 & 0 & 0 \\
            &   &   & 1 & 0 \\
            &   &   &   & 1
		\end{bmatrix}.
\end{equation}

(1) The homomorphism $(h,h')\colon G\to U_3\times U_3$ lifts to $U_5$ if and only if
the image of $t\cup t'$ in $H^2(G,P)$ vanishes. Since
\[P\simeq N\otimes N'\simeq \on{Ind}_{H}^G(\Z/2\Z)^{\otimes 2}\simeq \on{Ind}_{H}^G((\Z/2\Z)^{\oplus 2}),\]
we have the Faddeev--Shapiro isomorphism
\[H^2(G,P)\xrightarrow{\on{Res}} H^2(H,P)\to H^2(H,(\Z/2\Z)^{\oplus 2}),\]
and hence the restriction $H^2(G,P)\to H^2(H,P)$ is injective.
Therefore, $(h,h')$ lifts to $U_5$ if and only if the image of $t\cup t'$ in $H^2(H,P)$ vanishes. In view of (\ref{restrict-obstruction}), this is equivalent to
\[\varphi\cup\psi=\varphi\cup \on{Res}^G_H(\chi_3)=\on{Res}^G_H(\chi_2)\cup\psi=\on{Res}^G_H(\chi_2\cup\chi_3)=0.\] 
This implies that $\varphi\cup\psi=\varphi\cup \on{Res}^G_H(\chi_3)=0$ in $H^2(H,\Z/2\Z)$, but it is in fact equivalent to this. Indeed, if $\varphi\cup\psi=\varphi\cup \on{Res}^G_H(\chi_3)=0$, then by the projection formula we have \[\chi_2\cup\chi_3=\on{Cor}^G_{H}(\varphi)\cup\chi_3=\on{Cor}^G_{H}(\varphi\cup \on{Res}^G_H(\chi_3))=0,\] and hence $\on{Res}^G_H(\chi_2\cup\chi_3)=0$. Moreover, letting $g\in G-H$, we have \[\varphi\cup g(\psi)=\varphi\cup \on{Res}^G_H(\chi_3) + \varphi\cup \psi = 0+0=0,\]
and hence
\[\on{Res}^G_H(\chi_2)\cup\psi=\varphi\cup\psi+g(\varphi)\cup\psi=\varphi\cup\psi+g(\varphi\cup g(\psi))=0+0=0.\]
  
(2) The left vertical sequence in (\ref{cocycle-diag}) induces an exact sequence
		\begin{equation}\label{p-exact}
		H^2(G,\Z/2\Z) \to H^2(G,P)\to H^2(G,\cl{P}).
		\end{equation}
The homomorphism $(h,h')\colon G\to U_3\times U_3$ lifts to a homomorphism $G\to \cl{U}_5$ if and only if the pullback of the middle row of (\ref{cocycle-diag}) along $(h,h')$ splits, that is, if and only if the corresponding element in $H^2(G,\cl{P})$ is trivial. In view of (\ref{p-exact}), this happens if and only if the image of $t\cup t'$ in $H^2(G,P)$ belongs to the image of $H^2(G,\Z/2\Z)\to H^2(G,P)$. We have a commutative square
\[
\begin{tikzcd}
    H^2(G,\Z/2\Z) \arrow[d,"\on{Res}"] \arrow[r] & H^2(G,P) \arrow[d,"\on{Res}",hook] \\
    H^2(H,\Z/2\Z) \arrow[r] & H^2(H,P),
\end{tikzcd}
\]
where, as explained earlier, the right vertical map is injective by the Faddeev--Shapiro lemma. Therefore the image of $t\cup t'$ in $H^2(G,P)$ belongs to the image of $H^2(G,\Z/2\Z)\to H^2(G,P)$ if and only if the image of $t\cup t'$ in $H^2(H,P)$ belongs to the image of the composition $H^2(G,\Z/2\Z)\to H^2(H,\Z/2\Z)\to H^2(H,P)$. By (\ref{restrict-obstruction}), this is equivalent to the existence of $u\in H^2(G,\Z/2\Z)$ such that
\[\varphi\cup\psi=u,\qquad \varphi\cup \on{Res}^G_H(\chi_3)=\on{Res}^G_H(\chi_2)\cup\psi=\on{Res}^G_H(\chi_2\cup \chi_3)=0.\]
This implies that $\varphi\cup\psi$ comes from $H^2(G,\Z/2\Z)$ and that $\varphi\cup \on{Res}^G_H(\chi_3)=0$, but it is in fact equivalent to this. Indeed, given $u\in H^2(G,\Z/2\Z)$ such that $\varphi\cup\psi=u$ and $\varphi\cup \on{Res}^G_H(\chi_3)=0$, by the projection formula we have $\chi_2\cup\chi_3=0$, and hence $\on{Res}^G_H(\chi_2\cup\chi_3)=0$. Moreover, letting $g\in G-H$, we have 
\[\varphi\cup g(\psi)=\varphi\cup \on{Res}^G_H(\chi_3)+\varphi\cup\psi=0+u=u,\] and hence
\[\on{Res}^G_H(\chi_2)\cup\psi=\varphi\cup\psi+g(\varphi)\cup\psi=\varphi\cup\psi+g(\varphi\cup g(\psi))=u+g(u)=0.\qedhere\]
  \end{proof}
		
Using \Cref{u3u3}, we complete the proof of \Cref{u5-criterion} as follows. Suppose first that the Massey product $\ang{\chi_1,\chi_2,\chi_3,\chi_1}$ vanishes (resp. is defined). By \Cref{dwyer}, the homomorphism $(\chi_1,\chi_2,\chi_3,\chi_1)\colon G\to (\Z/2\Z)^4$ lifts to a homomorphism $G\to U_5$ (resp. $G\to\cl{U}_5$) and hence to a homomorphism $(h,h')\colon G\to U_3\times U_3$ given by matrices as in (\ref{h-h'-matrices}). We let $\varphi,\psi\in H^1(H,\Z/2\Z)$ be such that $\rho|_{H}=\varphi$ and $(\rho')|_{H}=\psi$. By \Cref{p-u3}(2), we have $\on{Cor}^G_{H}(\varphi)=\chi_2$ and $\on{Cor}^G_{H}(\psi)=\chi_3$ in $H^1(G,\Z/2\Z)$. By construction, $(h,h')$ lifts to $U_5$ (resp. $\cl{U}_5$), hence \Cref{u3u3} implies that $\varphi\cup \psi=\varphi\cup \on{Res}^G_H(\chi_3)=0$ in $H^1(H,\Z/2\Z)$ (resp. $\varphi\cup \psi$ comes from $H^2(G,\Z/2\Z)$ and $\varphi\cup \on{Res}^G_H(\chi_3)=0$ in $H^1(H,\Z/2\Z)$).

  Conversely, let $\varphi,\psi\in H^1(H,\Z/2\Z)$ be such that $\on{Cor}^G_{H}(\varphi)=\chi_2$, $\on{Cor}^G_{H}(\psi)=\chi_3$, and $\varphi\cup \psi=\varphi\cup \on{Res}^G_H(\chi_3)=0$ in $H^1(H,\Z/2\Z)$ (resp. $\varphi\cup \psi$ comes from $H^2(G,\Z/2\Z)$ and $\varphi\cup \on{Res}^G_H(\chi_3)=0$ in $H^1(H,\Z/2\Z)$). By \Cref{p-u3}(2), there exist homomorphisms $h,h'\colon G\to U_3$ as in (\ref{h-h'-matrices}) such that $\rho|_{H}=\varphi$ and $(\rho')|_{H}=\psi$. By \Cref{u3u3}, the homomorphism $(h,h')$ lifts to $U_5$ (resp. $\cl{U}_5$). On the other hand, the homomorphism $(h,h')\colon G\to U_3\times U_3$ is a lift of $(\chi_1,\chi_2,\chi_3,\chi_1)\colon G\to (\Z/2\Z)^4$. Therefore $(\chi_1,\chi_2,\chi_3,\chi_1)$ lifts to $U_5$ (resp. $\cl{U}_5$). By \Cref{dwyer}, the Massey product $\ang{\chi_1,\chi_2,\chi_3,\chi_1}$ is defined. 
\end{proof}

\begin{prop}\label{pi-lambda-res}
    Let $(G,\theta)$ be a $2$-oriented profinite group, let $M$ be a Hilbert 90 module for $(G,\theta)$, let $a,b\in M^G$, and suppose that $\chi_a\neq 0$ and $\chi_b\neq 0$. Let $\alpha\in M^{G_a}$ be such that $N_{G/G_a}(\alpha)=b$ and $\beta\in M^{G_b}$ such that $N_{G/G_b}(\beta)=a$.

    Then we have an exact sequence of $G$-modules
\[
M^{(G_a)_{\alpha}}\oplus M^{G_b} \xrightarrow{\pi} M^{G_a}  \xrightarrow{\lambda} H^2(G,M)\xrightarrow{\on{Res}} H^2((G_b)_{\beta},M),
\]
where \[\pi(\zeta,\eta) = N_{G_a/(G_a)_{\alpha}}(\zeta)-N_{G/G_{b}}(\eta)\] and $\lambda(\delta)$ is the image of $\on{Cor}^G_{G_a}(\chi_{\alpha}\cup\chi_{\delta})\in H^2(G,\Z/2\Z)$ in $H^2(G,M)$.
\end{prop}

\begin{proof}
    We define the following elements of $U_3$:
\[
\sigma_1\coloneqq \begin{bmatrix}
    1 & 1 & 0\\
    0 & 1 & 0\\
    0 & 0 & 1
\end{bmatrix},\qquad
\sigma_2\coloneqq \begin{bmatrix}
    1 & 1 & 0\\
    0 & 1 & 0\\
    0 & 0 & 1
\end{bmatrix},\qquad
\tau\coloneqq [\sigma_1,\sigma_2]= \begin{bmatrix}
    1 & 0 & 1\\
    0 & 1 & 0\\
    0 & 0 & 1
\end{bmatrix}.
\]
We have the following exact sequence of $U_3$-modules:
\begin{equation}\label{resolution}
0 \to \Z[U_3/\langle\sigma_2,\tau\rangle] \xrightarrow{f} \Z[U_3/\langle\sigma_2\rangle] \oplus \Z[U_3/\langle\sigma_1,\tau\rangle] \xrightarrow{g} \Z[U_3/\langle\sigma_1\rangle] \oplus \Z \xrightarrow{k} \Z \to 0,
\end{equation}
where
\begin{align*}
f(1) & = (1+\tau, -1-\sigma_2), \\
g(1,0) & = (1+\sigma_2, -1), \\
g(0,1) & = (1+\tau, -1), \\
k(1,0) & = 1, \\
k(0,1) & = 2.
\end{align*}
By \Cref{p-u3}(2), there exists a continuous homomorphism $h\colon G\to U_3$ given by a matrix (\ref{a-b-rho}) such that $\rho|_{G_a}=\chi_{\alpha}$. Since $\chi_a$ and $\chi_b$ are non-zero, either $h$ is surjective or the image of $h$ is the subgroup of $U_3$ generated by $\sigma_1\sigma_2$. In either case, we have
\begin{equation}\label{correct-index}
    [G:G_a]=[G:G_b]=2,\qquad [G:(G_a)_\alpha]=[G:(G_b)_\beta]=4.
    \end{equation}
We regard (\ref{resolution}) as an exact sequence of $G$-modules via $h$. In view of (\ref{correct-index}), an application of the functor $\on{Hom}_{\Z}(-,M)$ yields the exact sequence of $G$-modules
\begin{equation}\label{mu-epsilon-pi}
0 \to M \xrightarrow{\mu} \on{Ind}_{(G_b)_{\beta}}^G(M)\oplus M \xrightarrow{\varepsilon} \on{Ind}_{(G_a)_{\alpha}}^G(M) \oplus \on{Ind}_{G_b}^G(M)\xrightarrow{\pi} \on{Ind}_{G_a}^G(M) \to 0,
\end{equation}
where 
\[\pi(\zeta,\eta) = N_{G_a/(G_a)_{\alpha}}(\zeta)-N_{G/G_{b}}(\eta).\]
Let $T = \on{Ker}(\pi) = \on{Coker}(\mu)$. Since $M$ is a Hilbert 90 module for $G$, by \Cref{subgroup} it is also a Hilbert 90 module for the open subgroups $(G_a)_{\alpha}$, $G_b$ and $(G_b)_{\beta}$, and hence by Faddeev--Shapiro's lemma
\begin{align*}
    H^1(G,\on{Ind}_{(G_b)_{\beta}}^G(M)\oplus M)=H^1((G_b)_{\beta},M)\oplus H^1(G,M)=0, \\
    H^1(G,\on{Ind}_{(G_a)_{\alpha}}^G(M) \oplus \on{Ind}_{G_b}^G(M))=H^1((G_a)_{\alpha},M)\oplus H^1(G_b,M)=0.
\end{align*}
Thus (\ref{mu-epsilon-pi}) yields the exact sequences
\[
M^{(G_a)_{\alpha}}\oplus M^{G_b} \xrightarrow{\pi} M^{G_a} \to H^1(G,T) \to 0
\]
and
\[
0\to H^1(G,T) \to H^2(G,M)\xrightarrow{\on{Res}} H^2((G_b)_{\beta},M).
\]
We obtain an exact sequence of $G$-modules
\[
M^{(G_a)_{\alpha}}\oplus M^{G_b} \xrightarrow{\pi} M^{G_a}  \xrightarrow{\lambda'} H^2(G,M)\xrightarrow{\on{Res}} H^2((G_b)_{\beta},M),
\]
where $\lambda'$ is defined as the composition $M^{G_a}\to H^1(G,T)\to H^2(G,M)$. It remains to show that $\lambda'=\lambda$. 

Let $\delta\in M^{G_a}$. By \Cref{serre-brauer}, the homomorphism $H^2(G_a,\Z/2\Z)\to H^2(G_a,M)$ sends $\chi_{\alpha}\cup \chi_{\delta}$ to $\omega_{\alpha}\cup\delta$, where $\omega_{\alpha}\in H^2(G_a,\Z)$ is defined as the Bockstein of $\chi_{\alpha}\in H^1(G_a,\Z/2\Z)$. Thus $\lambda(\delta)=\on{Cor}^G_{G_a}(\omega_a\cup\delta)$. 

Consider the following commutative diagram
\begin{equation}\label{three-squares-u3}
\begin{adjustbox}{max width=\textwidth}
\begin{tikzcd}
0 \arrow[r] & \Z[U_3/\langle\sigma_2,\tau\rangle] \arrow[d,equal] \arrow[r,"f"] & \Z[U_3/\langle\sigma_2\rangle] \oplus \Z[U_3/\langle\sigma_1,\tau\rangle] \arrow[d,"\text{$(1,0)$}"] \arrow[r,"g"] & \Z[U_3/\langle\sigma_1\rangle] \oplus \Z \arrow[r,"k"] \arrow[d,"l"] & \Z \arrow[d,"1+\sigma_1"] \arrow[r] & 0 \\
0\arrow[r] & \Z[U_3/\ang{\sigma_2,\tau}] \arrow[r,"1+\tau"] &  \Z[U_3/\ang{\sigma_2}] \arrow[r,"1-\tau"] & \Z[U_3/\ang{\sigma_2}] \arrow[r,"1"] & \Z[U_3/\ang{\sigma_2,\tau}] \arrow[r] & 0,
\end{tikzcd}
\end{adjustbox}
\end{equation}
where the top row is (\ref{resolution}) and
\begin{align*}
    l(1,0)&= 1+\sigma_1,\\
    l(0,1)&=1+\sigma_1+\tau+\sigma_1\tau.
\end{align*}
The commutativity of (\ref{three-squares-u3}) is easily verified using the explicit descriptions of the arrows. By applying the functor $\on{Hom}_{\Z}(-,M)$ to each term of (\ref{three-squares-u3}), we obtain the following commutative diagram with exact rows:
\begin{equation}\label{three-squares-g}
\begin{adjustbox}{max width=\textwidth}
\begin{tikzcd}
0 \arrow[r] & \on{Ind}_{G_a}^G(M) \arrow[d,"\nu"] \arrow[r] & \on{Ind}_{(G_a)_{\alpha}}^G(M) \arrow[r,"1-\tau"] \arrow[d] & \on{Ind}_{(G_a)_{\alpha}}^G(M) \arrow[r,"\nu'"] \arrow[d,"\text{$(1,0)$}"] & \on{Ind}_{G_a}^G(M) \arrow[d,equal] \arrow[r] & 0 \\  
0 \arrow[r] & M  \arrow[r,"\mu"] & \on{Ind}_{(G_b)_{\beta}}^G(M)\oplus M \arrow[r,"\varepsilon"] & \on{Ind}_{(G_a)_{\alpha}}^G(M) \oplus \on{Ind}_{G_b}^G(M) \arrow[r,"\pi"] & \on{Ind}_{G_a}^G(M) \arrow[r] & 0.
\end{tikzcd}
\end{adjustbox}
\end{equation}
Here $\nu$ is the map of (\ref{ind-norm}) for $H=G_a$ and $A=M$, and $\nu'$ is obtained by applying the functor $\on{Ind}^G_{G_a}$ to the map of (\ref{ind-norm}) for the profinite group $G_a$, the open subgroup $H=(G_a)_{\alpha}$ and the $G_a$-module $A=\on{Ind}^G_{G_a}(M)$.

Let $N\coloneqq \ang{\sigma_2,\tau}$. The bottom row of (\ref{three-squares-u3}) is induced from the exact sequence of $N$-modules
\[0\to \Z\to \Z[N/\ang{\sigma_2}]\xrightarrow{1-\tau} \Z[N/\ang{\sigma_2}]\to \Z\to 0.\] 
This sequence may be viewed as a sequence of $N/\ang{\sigma_2}$-modules. As $N/\ang{\sigma_2}$ is cyclic of order $2$, by \cite[Proposition 3.4.11]{gille2017central} the induced map $\Z\to H^2(N/\ang{\sigma_2},\Z)$ sends $1$ to the generator of $H^2(N/\ang{\sigma_2},\Z/2\Z)$, that is, to the Bockstein of the generator of $H^1(\Z,\Z/2\Z)$. It follows that the map $M^G\to H^2(G_a,M)$ induced by the top row of (\ref{three-squares-g}) is given by cup product with the Bockstein of $\chi_a$, that is, with $\omega_a$. Therefore (\ref{three-squares-g}) gives a commutative square 
\begin{equation}\label{lambda'-square}
\begin{tikzcd}
    M^{G_a} \arrow[r,"(-)\cup\omega_{\alpha}"] \arrow[d,equal] & H^2(G_a,M) \arrow[d,"\on{Cor}^G_{G_a}"] \\
    M^{G_a} \arrow[r,"\lambda'"] & H^2(G,M).
\end{tikzcd}
\end{equation}
We conclude that $\lambda'(\delta)= \on{Cor}^G_{G_a}(\omega_{\alpha}\cup\delta)=\lambda(\delta)$, as desired.\end{proof}

 \begin{proof}[Proof of \Cref{main-h90-degenerate}]
By \Cref{trivial-case}, we may assume that $\chi_1\neq 0$ and $\chi_2\neq 0$. By \Cref{hilbert-90-equiv}, there exists a Hilbert 90 module $M$ for $(G,\theta)$. Let $\iota\colon \Z/2\Z\hookrightarrow M$ be the unique injective homomorphism, and let $a,b,c\in M^G$ be such that $\chi_1=\chi_a$, $\chi_2=\chi_b$ and $\chi_3=\chi_c$; see below (\ref{connecting}) for the notation. Since $\chi_a$ and $\chi_b$ are non-trivial, by \Cref{u5-criterion}(2) there exist $\alpha,\delta \in M^{G_a}$ such that $\on{Cor}_{G/G_a}(\chi_{\alpha})=\chi_b$, $\on{Cor}_{G/G_a}(\chi_{\delta})=\chi_c$,  $\chi_{\alpha}\cup \chi_{\delta}$ belongs to the image of $H^2(G,\Z/2\Z)\to H^2(G_a,\Z/2\Z)$ and $\chi_{\alpha}\cup \chi_c=0$ in $H^2(G_a,\Z/2\Z)$.

Since $\chi_{\alpha}\cup\chi_{\delta}$ comes from $H^2(G,\Z/2\Z)$, we have $\on{Cor}^G_{G_a}(\chi_{\alpha}\cup\chi_{\delta})=0$. As $\chi_a$ and $\chi_b$ are non-trivial,  \Cref{pi-lambda-res} implies that $\lambda(\delta)=0$. It follows that $\delta$ belongs to the image the map $\pi$ appearing in \Cref{pi-lambda-res}, that is, $\delta = N_{G_a/(G_a)_{\alpha}}(\zeta)-N_{G/G_b}(\eta)$ for some $\zeta \in M^{(G_a)_{\alpha}}$ and $\eta \in M^{G_b}$. Letting $z \coloneqq N_{G/G_b}(\eta)\in M^G$, we obtain \[\chi_{\alpha}\cup\chi_{\delta} = \chi_{\alpha}\cup\on{Cor}^{G_a}_{(G_a)_{\alpha}}(\zeta)+\chi_{\alpha}\cup\chi_z= \chi_{\alpha}\cup\chi_z\qquad\text{in $H^2(G_a,\Z/2\Z)$.}\]
It follows that $\chi_{\alpha}\cup\chi_{\delta+z}=0$ in $H^2(G_a,\Z/2\Z)$. Moreover, \[N_{G/G_a}(\delta+z)=N_{G/G_a}(\delta)+2z=c+0=c,\] and so $\on{Cor}^G_{G_a}(\chi_{\delta+z})=\chi_c$. We already know that $\chi_a\neq 0$, $\on{Cor}^G_{G_a}(\chi_\alpha)=\chi_b$ and $\chi_{\alpha}\cup\chi_c=0$ in $H^1(G_a,\Z/p\Z)$, the elements $\alpha$ and $\delta+z$ satisfy the conditions of \Cref{u5-criterion}(2). By \Cref{u5-criterion}(2), the Massey product $\ang{\chi_1,\chi_2,\chi_3,\chi_1}$ vanishes.
 \end{proof}

\section{An example}\label{section-example}

In this section, we illustrate \Cref{main-h90} with an example. Consider the group $G$ given by the following presentation by generators and relations:
\begin{equation}\label{g-define}G\coloneqq\ang{a,b: a^2b=ba^2}.\end{equation}
As in the proof of \Cref{pi-lambda-res}, we define the following elements of $U_3\subset \on{GL}_3(\F_2)$:
\[
\sigma_1\coloneqq 
\begin{bmatrix}
    1 & 1 & 0 \\
    0 & 1 & 0 \\
    0 & 0 & 1
\end{bmatrix},\qquad
\sigma_2\coloneqq 
\begin{bmatrix}
    1 & 0 & 0 \\
    0 & 1 & 1 \\
    0 & 0 & 1
\end{bmatrix},\qquad
\tau\coloneqq [\sigma_1,\sigma_2]=
\begin{bmatrix}
    1 & 0 & 1 \\
    0 & 1 & 0 \\
    0 & 0 & 1
\end{bmatrix}.
\]
    The group $G$ fits into the diagram with exact rows
    \begin{equation}\label{h-define}
    \begin{tikzcd}
        1 \arrow[r] & H \arrow[r]\arrow[d,"\varphi|_H"] & G \arrow[r,"f"]\arrow[d,"\varphi"]  & \Z/2\Z \arrow[r] \arrow[d,equal] & 1 \\
        1 \arrow[r] & N' \arrow[r] & U_3 \arrow[r,"u_{2,3}"] & \Z/2\Z \arrow[r] & 1,
    \end{tikzcd}
    \end{equation}
    where $\varphi$ is determined by $\varphi(a)=\sigma_1$ and $\varphi(b)=\sigma_2$, the subgroup $H\subset G$ is defined as the kernel of $f$, and the subgroup $N'\subset U_3$ as the kernel of the coordinate map $u_{2,3}$.

\begin{prop}\label{counterexample-formal-h90}
    Let $\hat{G}$ be the pro-$2$ completion of $G$ defined in (\ref{g-define}) and let $\hat{H}\subset \hat{G}$ be the pro-$2$ completion of the subgroup $H\subset G$ considered in (\ref{h-define}).

    (1) The reduction maps $H^1(\hat{G},\Z/2^n\Z)\to H^1(\hat{G},\Z/2\Z)$ are surjective for all $n\geq 1$.

    (2) The reduction map $H^1(\hat{H},\Z/4\Z)\to H^1(\hat{H},\Z/2\Z)$ is not surjective.

    (3) There exist $\chi_1,\chi_2,\chi_3\in H^1(\hat{G},\Z/2\Z)$ such that $\ang{\chi_1,\chi_2,\chi_3}$ is defined but does not vanish.
\end{prop}

\begin{proof}[Proof of \Cref{counterexample-formal-h90}(1) and (2)]
(1) We have $G/[G,G]\simeq \Z\times \Z$, and hence the reduction maps $H^1(G,\Z/2^n\Z)\to H^1(G,\Z/2\Z)$ are surjective for all $n\geq 1$. By the universal property of pro-$2$ completion, the canonical map $G\to\hat{G}$ induces isomorphisms $H^1(\hat{G},\Z/2^n\Z)\to H^1(G,\Z/2^n\Z)$ for all $n\geq 1$. It follows that the reduction maps $H^1(\hat{G},\Z/2^n\Z)\to H^1(\hat{G},\Z/2\Z)$ are also surjective for all $n\geq 1$.

(2)  We have $N'\simeq (\Z/2\Z)^2$ and the map $u_{1,3}\colon N'\to \Z/2\Z$ is a homomorphism. Define $\chi\coloneqq u_{1,3}\circ \varphi|_H \in H^1(H,\Z/2\Z)$. Suppose by contradiction that $\chi$ lifts to some $\widetilde{\chi}\in H^1(H,\Z/4\Z)$.

Let $h\coloneqq aba^{-1}b^{-1}\in H$. Then $\varphi(h)=\tau$ in $N'$, and hence $\chi(h)=1$ in $\Z/2\Z$. Note that $a$ belongs to $H$, and let $c\coloneqq bab^{-1}\in H$, so that we have the identities $h=ac^{-1}$ and $c^2=ba^{2}b^{-1}=a^2$ in $H$. Since the group $\Z/4\Z$ is abelian, we have $\widetilde{\chi}(h^2)=\widetilde{\chi}(a^2c^{-2})=0$ in $\Z/4\Z$.
It follows that $\widetilde{\chi}(h)$ belongs to $2\Z/4\Z$, and hence $\chi(h)=0$, a contradiction.
\end{proof}

\begin{lemma}\label{i+n}
    Let $p$ be a prime and $n\geq 3$ be an integer. Let $N\in \on{Mat}_{n+1}(\F_p)$ be such that $N_{i,j}=1$ if $j=i+1$ and $N_{i,j}=0$ elsewhere. 

    (1) The $U_{n+1}$-centralizer of $I+N$ consists of all upper unitriangular matrices which are constant on all upper diagonals.
    
    (2) The $U_{n+1}$-conjugacy class of $I+N$ consists of all upper unitriangular matrices $M$ with entry $1$ at $(i,i+1)$ for all $i$.
\end{lemma}

\begin{proof}
(1) A matrix $M$ is in the $U_{n+1}$-centralizer of $I+N$ if and only if $MN=NM$. This is equivalent to $M_{i,j+1}=M_{i-1,j}$ for all $i,j$.

(2)    The size of the $U_{n+1}$-conjugacy class of $I+N$ is equal to the index of the centralizer of $I+N$ in $U_{n+1}$. The order of $U_{n+1}$ is $p^{\frac{n(n+1)}{2}}$. By (1), the order of the centralizer of $I+N$ is equal to $p^n$, thus the size of the $U_{n+1}$-conjugacy class of $I+N$ is equal to $p^{\frac{n(n-1)}{2}}$.

If $M$ is conjugate to $I+N$, then the image of $M$ in $U_{n+1}/[U_{n+1},U_{n+1}]$ is equal to the image of $I+N$. This implies that $M_{i,i+1}=(I+N)_{i,i+1}=1$ for all $i$. This restricts the possible conjugates of $I+N$ to at most $p^{\frac{n(n-1)}{2}}$ elements, hence by the previous paragraph they must form the conjugacy class of $I+N$.
\end{proof}

\begin{proof}[Proof of \Cref{counterexample-formal-h90}(3)]
Let
    \[
    A\coloneqq I+N=\begin{bmatrix}
        1 & 1 & 0 & 0\\
        0 & 1 & 1 & 0 \\
        0 & 0 & 1 & 1 \\
        0 & 0 & 0 & 1
    \end{bmatrix}\in U_4,
    \]
    so that
    \[A^2=I+N^2=\begin{bmatrix}
        1 & 0 & 1 & 0\\
        0 & 1 & 0 & 1 \\
        0 & 0 & 1 & 0 \\
        0 & 0 & 0 & 1
    \end{bmatrix}.\]
    We write $\cl{A}$ for the image of $A$ in $\cl{U}_4$.
    Let 
    \[B=\begin{bmatrix}
        1 & B_{1,2} & B_{1,3} & B_{1,4}\\
        0 & 1 & B_{2,3} & B_{2,4} \\
        0 & 0 & 1 & B_{3,4} \\
        0 & 0 & 0 & 1
    \end{bmatrix}\in U_4.
    \]
    Then
    \[BN^2=\begin{bmatrix}
        0 & 0 & 1 & B_{1,2}\\
        0 & 0 & 0 & 1 \\
        0 & 0 & 0 & 0 \\
        0 & 0 & 0 & 0
    \end{bmatrix},\qquad
    N^2B=\begin{bmatrix}
        0 & 0 & 1 & B_{3,4}\\
        0 & 0 & 0 & 1 \\
        0 & 0 & 0 & 0 \\
        0 & 0 & 0 & 0
    \end{bmatrix}.
    \]
    Since $A^2=I+N^2$, we deduce that:

    (i) $\cl{A}^2$ belongs to the center of $\cl{U}_4$, and
    
    (ii) $A^2$ commutes with $B$ in $U_4$ if and only if $B_{1,2}=B_{3,4}$.

     In view of \Cref{dwyer} and the universal property of pro-$2$ completion, in order to prove (3) it suffices to find a homomorphism $G\to (\Z/2\Z)^3$ which lifts to $\cl{U}_4$ but not to $U_4$. If $S$ is a group, to give a homomorphism $G\to S$ is the same as giving two elements $s_a,s_b\in S$ such that $s_bs_a^2=s_a^2s_b$. 

    Consider the homomorphism $\chi\colon G\to (\Z/2\Z)^3$ given by sending $a$ to $(1,1,1)$ and $b$ to $(1,0,0)$. 
    By (i), sending
    \[
    a\mapsto \cl{A},\qquad b\mapsto \begin{bmatrix}
        1 & 1 & 0 & \square\\
        0 & 1 & 0 & 0 \\
        0 & 0 & 1 & 0 \\
        0 & 0 & 0 & 1
    \end{bmatrix},
    \]
    defines a lift of $\chi$ to $\cl{U}_4$.
    
    Suppose that $\chi$ lifts to a group homomorphism $f\colon G\to U_4$. By \Cref{i+n}(2), $f(a)$ is $U_4$-conjugate to $A$, and so, replacing $f$ by a $U_4$-conjugate if necessary, we may suppose that $f(a)=A$. On the other hand, we have
    \[f(b)=\begin{bmatrix}
        1 & 1 & x_1 & x_3\\
        0 & 1 & 0 & x_2 \\
        0 & 0 & 1 & 0 \\
        0 & 0 & 0 & 1
    \end{bmatrix},
    \]
    for some $x_1,x_2,x_3\in \Z/2$. Since $f(b)_{1,2}=1\neq 0=f(b)_{3,4}$, we deduce from (i) that $f(b)$ does not commute with $A^2=f(a^2)$, contradicting the fact that $b$ commutes with $a^2$. Therefore $\chi$ does not lift to $U_4$.
\end{proof}

\newcommand{\etalchar}[1]{$^{#1}$}

\end{document}